\documentclass[11pt]{amsart}

\usepackage{amscd}
\usepackage{amsmath, amssymb}
\usepackage{amsfonts}
\usepackage{enumerate}
\newcommand{\vp}{\varphi}
\newcommand{\ve}{\varepsilon}
\newcommand{\ddbar}{\sqrt{-1} \partial \overline{\partial}}
\newcommand{\ol}{\overline}

\begin{document}
\newcounter{remark}
\newcounter{theor}
\setcounter{remark}{0}
\setcounter{theor}{1}
\newtheorem{claim}{Claim}
\newtheorem{theorem}{Theorem}[section]
\newtheorem{lemma}[theorem]{Lemma}
\newtheorem{corollary}[theorem]{Corollary}
\newtheorem{proposition}[theorem]{Proposition}
\newtheorem{question}{question}[section]
\newtheorem{defn}{Definition}[theor]
\numberwithin{equation}{section}

\title{ The $2$-nd  Hessian  type  equation on  almost Hermitian manifolds}
\author[J. Chu]{Jianchun Chu}
\address{Institute of Mathematics, Academy of Mathematics and Systems Science, Chinese Academy of Sciences, Beijing 100190, P. R. China}
\email{chujianchun@gmail.com}
\author[L. Huang]{Liding Huang}
\address{School of Mathematical Sciences, University of Science and Technology of China, Hefei 230026, P. R. China}
\email{hldliding@sina.com}
\author[X. Zhu]{Xiaohua ${\rm Zhu}^*$}
\address{School of Mathematical Sciences, Peking University, Yiheyuan Road 5, Beijing 100871, P. R. China }
\email{xhzhu@math.pku.edu.cn}

\thanks {* Partially supported by the NSFC Grants  11331001}

\begin{abstract}
In this paper, we derive the second order estimate to  the  $2$-nd  Hessian  type equation  on a  compact almost Hermitian manifold.
\end{abstract}
\maketitle

\section{Introduction}

As a generalization of  Laplace  equation and complex Monge-Amp\`{e}re equation on  a  complex manifold $M$,
the following  $k$-th  complex Hessian equation ($1<k<n$) has been studied extensively,
\begin{equation}\label{k-th complex Hessian equation}
\left\{\begin{array}{ll}
\ (\omega+\ddbar\vp)^{k}\wedge\omega^{n-k}=e^{F}\omega^{n}\\[1mm]
\ \omega+\ddbar\vp\in\Gamma_{k}(M)\\[1mm]
\ \sup_{M}\vp=0,
\end{array}\right.
\end{equation}
where $\Gamma_{k}(M)$ is the  space of  $k$-th convex  $(1,1)$-forms (cf. Section 2).  When $(M,\omega)$ is a compact K\"{a}hler manifold,
the second order a priori estimate was obtained by Hou \cite{Hou09}  and Hou-Ma-Wu \cite{HMW10}.  Lately,  by using  Hou-Ma-Wu's result,   Dinew-Ko{\l}odziej \cite{DK12}  solved the existence of   (\ref{k-th complex Hessian equation}).  Sz\'ekelyhidi \cite{Sze15}  extended  Dinew-Ko{\l}odziej's result     to a  Hermitian  manifold   (see also \cite{Zha15} by Zhang).

The $2$-nd complex Hessian type  equation plays an important role in  Strominger system from  the string theory \cite{Str86}. In \cite{FY08}, Fu-Yau reduced the Strominger system to an equation
\begin{equation}\label{Fu-Yau equation 1}
\ddbar(e^{\vp}-fe^{-\vp})\wedge\omega^{n-1}+n\alpha\ddbar\vp\wedge\ddbar\vp\wedge\omega^{n-2}+\mu\frac{\omega^{n}}{n!}=0,
\end{equation}
where $\alpha\in\mathbb{R}$ is a slope parameter  and $f,\mu\in C^{\infty}(M)$ satisfy  some admissible conditions.  They found that
(\ref{Fu-Yau equation 1}) can be written  as  a general $2$-nd Hessian equation,
\begin{equation}\label{Fu-Yau equation 2}
\left\{\begin{array}{ll}
\ ((e^{\vp}+fe^{-\vp})\omega+2n\alpha\ddbar\vp)^{2}\wedge\omega^{n-2}=e^{F(z,\partial\vp,\vp)}\omega^{n}\\[1mm]
\ (e^{\vp}+fe^{-\vp})\omega+2n\alpha\ddbar\vp\in\Gamma_{2}(M),\\[1mm]
\end{array}\right.
\end{equation}
where
\begin{equation*}
\begin{split}
e^{F(z,\partial\vp,\vp)} & = e^{2\vp}(1-4\alpha e^{-\vp}|\partial\vp|_{g}^{2})+4\alpha fe^{-\vp}|\partial\vp|_{g}^{2}+2f\\
                      & \quad +e^{-2\vp}f^{2}-\frac{4\alpha\mu}{n-1}+4\alpha e^{-\vp}\left(\Delta f-2\textrm{Re}(g^{i\overline{j}}f_{i}\vp_{\overline{j}})\right).
\end{split}
\end{equation*}
By (\ref{Fu-Yau equation 2}),  Fu-Yau \cite{FY07,FY08} solved (\ref{Fu-Yau equation 1}) on a toric fibration over a K3 surface.  Recently, Phong-Picard-Zhang \cite{PPZ15} obtained  a priori estimates of (\ref{Fu-Yau equation 2}) with slope parameter $\alpha>0$ on a compact K\"{a}hler manifold. In \cite{PPZ16b}, they also solved the existence of (\ref{Fu-Yau equation 2})  with slope parameter $\alpha<0$.

In this paper, we  generalize  the $2$-nd complex Hessian equation  to an   almost Hermitian manifold $(M,\omega,J)$ and consider equation,
\begin{equation}\label{General 2-nd Hessian equation}
\left\{ \begin{array}{ll}
(\chi(z,\vp)+\ddbar\vp)^{2}\wedge\omega^{n-2}=e^{F(z,\partial\vp,\vp)}\omega^{n} \\[1mm]
( \chi(z,\vp)+\ddbar\vp) \in\Gamma_{2}(M),\\[1mm]
\end{array}\right.
\end{equation}
where $\chi(z,\vp)$ is a positive $(1,1)$-form which may depend on the solution $\varphi$.   We prove the following $C^2$-estimate.

\begin{theorem}\label{Generalized second order estimate}Let  $(M,\omega,J)$ be a compact   almost Hermitian manifold.
 Suppose that  $\chi(z,\varphi)\geq\ve_{0}\omega$ for a positive constant $\ve_{0}>0$ and $\vp$ is a smooth solution of (\ref{General 2-nd Hessian equation}). Then  the following estimate holds,
\begin{equation}\label{c2-varphi}
\sup_{M}|\nabla^{2}\vp|_{g}\leq C,
\end{equation}
where $\nabla$ is the Levi-Civita connection  of $g$ and $C$ is a uniform  constant depending only on $\ve_{0}$, $\|\vp\|_{C^{1}}$, $\|F\|_{C^{2}}$, $\|\chi\|_{C^{2}}$ and $(M,\omega,J)$.
\end{theorem}

We note that Theorem \ref{Generalized second order estimate} holds for any solution $\varphi$  with $(\chi(z,\vp)+\ddbar\vp)\in\Gamma_{2}(M)$ and we do not need to  assume that $\varphi$ is  $\chi(z,\vp)$-convex. When $M$ is K\"{a}hlerian, an analogy of (\ref{c2-varphi}) was obtained for some special function $F$ by Phong-Picard-Zhang \cite{PPZ15,PPZ16b}. In another paper  \cite{PPZ16a},  they also got similar  estimate  (\ref{c2-varphi}) for $\chi(z,\vp)$-convex solutions for  general $k$-th complex Hessian equation on a  K\"{a}hler manifold.

Compared to the work of Phong-Picard-Zhang \cite{PPZ15,PPZ16b}, our method is quite different. First, for general right hand side $F(z,\partial\vp,\vp)$, there are more troublesome terms when one differentiates the equation (\ref{General 2-nd Hessian equation}). We overcome this new obstacle by investigating the structure of $\log\sigma_{2}$ (see Lemma \ref{control C lambda1}). Second, since the almost complex structure $J$ may be not integrable, there are more "bad" third order terms. In order to deal with these terms, we need to analyse the concavity of the operator $\log\sigma_{2}$. More precisely, we estimate the eigenvalues and eigenvectors of the matrix $(-G^{i\overline{i},j\overline{j}})$ (see Lemma \ref{Lemma 6 and Lemma 7}). The structure of $\log\sigma_{2}$ plays an important role in the proof, which involves some delicate calculations. We expect that the analogous argument  can be extended  to study  $\log\sigma_{k}$ ($k>2$).

More recently,  Chu-Tosati-Weinkove \cite{CTW16} studied the Monge-Amp\`{e}re equation on compact almost Hermitian manifolds and  proved  the existence and uniqueness of solutions for generalized Calabi-Yau equation.  Since the manifold is just almost Hermitian, they  gave an approach  to estimate the  Hessian of solution instead of its complex Hessian.   Our motivation is from their work. In addition, the almost Hermitian manifold is a natural research object in non-K\"{a}hler geometry. The motivation of study is  from differential geometry as well as  mathematical physics.   We refer  the reader to  interesting  papers such as \cite{GSVY90,NH03,DeT06,ST12,HL15}, etc.

At present, our  computations  just work for (\ref{General 2-nd Hessian equation}), not available  for general  $k$-th complex Hessian equation.  On the other hand,   the constant $C$ in  (\ref{c2-varphi})   depends on  the norm of $\partial\vp$.   We hope that there exists a   $C^2$-estimate to  (\ref{General 2-nd Hessian equation})  which  may give an   explicit   dependence  on $\partial\vp$, so that it can be applied to study the existence of  (\ref{General 2-nd Hessian equation}) as in  \cite{PPZ16b}.

The organization of paper is as follows. In Section 2,  we introduce an  auxiliary function $\hat Q$  in order to estimate the largest eigenvalue of Hessian matrix of solution  of   $\sigma_2$-equation.  Then in Section 3, we  estimate the lower bound of $L(\hat Q)$  for the linear elliptic operator $L$ of $\sigma_2$-equation. The main estimate will be given in Section 4, where Theorem \ref{Generalized second order estimate} will be proved at the end. In Section 5, we give the proof of Lemma \ref{Lemma 6 and Lemma 7}.

\bigskip

{\bf Acknowledgments. }The first-named author would like to thank his advisor G. Tian for encouragement and support. He also thanks  V. Tosatti and B. Weinkove for their collaboration.

\section{Preliminaries}
\subsection{$\Gamma_k(M)$-space}

On an  almost Hermitian manifold  $(M,\omega,J)$ with  real dimension $2n$,      $\partial$ and $\overline{\partial}$ operators can be also defined for any $(p,q)$-form $\beta$ (cf. \cite{HL15,CTW16}).  In particular,  for any  $f\in C^{2}(M)$,  $\ddbar f=\frac{1}{2}(dJdf)^{(1,1)}$ is a real $(1,1)$-form in $A^{1,1}(M)$, where $A^{1,1}(M)$ is the space of smooth real (1,1) forms on $(M,\omega,J)$. Let $\{e_{i}\}_{i=1}^{n}$ be a local frame for $T_{\mathbb{C}}^{(1,0)}M$ associated to  Riemannian metric $g$ on $(M,\omega,J)$.   Then (cf. \cite[(2.5)]{HL15})
\begin{equation*}
f_{i\overline{j}}=(\ddbar f)(e_{i},\overline{e}_{j})=e_{i}\overline{e}_{j}(f)-[e_{i},\overline{e}_{j}]^{(0,1)}(f).
\end{equation*}

As usually,  we let  $\sigma_{k}$  ($1\leq k\leq n$)  and $\Gamma_{k}$ be  the $k$-th elementary symmetric function and   the $k$-th Garding cone on $\mathbb R^n$, respectively. Namely,
 for any $\eta=(\eta_{1},\eta_{2},\cdots,\eta_{n})\in\mathbb{R}^{n}$, we have
\begin{equation*}
\begin{split}
& \sigma_{k}(\eta) =\sum_{1<i_{1}<\cdots<i_{k}<n}\eta_{i_{1}}\eta_{i_{2}}\cdots\eta_{i_{n}},\\
\Gamma_{k}  = & \{ \eta\in\mathbb{R}^{n}~|~ \text{$\sigma_{j}(\eta)>0$ for $j=1,2,\cdots,k$} \}.
\end{split}
\end{equation*}
Clearly $\sigma_k$ is a $k$-multiple functional.  Then one can extend it to $A^{1,1}(M)$ by
\begin{equation*}
\sigma_{k}(\alpha)=\left(
\begin{matrix}
n\\
k
\end{matrix}
\right)
\frac{\alpha^{k}\wedge\omega^{n-k}}{\omega^{n}}, ~\forall ~\alpha\in A^{1,1}(M).
\end{equation*}
Define  a  cone $\Gamma_{k}(M)$ on   $A^{1,1}(M)$ by
\begin{equation*}
\Gamma_{k}(M)=\{ \alpha\in A^{1,1}(M)~|~ \text{$\sigma_{j}(\alpha)>0$ for $j=1,2,\cdots,k$} \}.
\end{equation*}
Thus  we   can  introduce  a $\sigma_k(\cdot)$ operator for any $\varphi\in C^\infty(M)$ with  $\tilde \omega= (\chi+\ddbar\vp)\in \Gamma_{k}(M)$  by
$$\sigma_k(\chi+\ddbar\vp),$$
where $\chi$ is a real $(1,1)$-form, which may depend on $\varphi$.

In this paper, we  are interested in $\sigma_2$ operator.  We use the following notation
\begin{equation*}
G^{i\overline{j}}=\frac{\partial\log\sigma_{2}(\tilde{\omega})}{\partial \tilde{g}_{i\overline{j}}} \text{~and~} G^{i\overline{j},k\overline{l}}=\frac{\partial^{2}\log\sigma_{2}(\tilde{\omega})}{\partial \tilde{g}_{i\overline{j}}\partial \tilde {g}_{k\overline{l}}},
\end{equation*}
where  $\tilde{g}_{i\overline{j}}=\chi_{i\overline{j}}+\vp_{i\overline{j}}$.
For any point $x_{0}\in M$, let $\{e_{i}\}_{i=1}^{n}$ be a local unitary frame (with respect to $g$) such that $\tilde{g}_{i\overline{j}}(x_{0})=\delta_{ij}\tilde{g}_{i\overline{i}}(x_{0})$. We denote $\tilde{g}_{i\overline{i}}(x_{0})$ by $\eta_{i}$
and assume
\begin{equation*}
\eta_{1}\geq\eta_{2}\geq\cdots\geq\eta_{n}.
\end{equation*}
  Then at $x_{0}$, we have
\begin{equation}\label{the first order derivative of F}
G^{i\overline{j}}=G^{i\overline{i}}\delta_{ij}=\frac{\sigma_{1}(\eta|i)}{\sigma_{2}(\eta)}\delta_{ij},
\end{equation}
where $\sigma_{1}(\eta|i) = \sum_{j\neq i}\eta_{j}$.   Also we  have
\begin{equation*}\label{the second order derivative of F}
G^{i\overline{j},k\overline{l}}=
\left\{\begin{array}{ll}
G^{i\overline{i},k\overline{k}}, \text{~~~~if $i=j$, $k=l$;}\\[1mm]
G^{i\overline{k},k\overline{i}}, \text{~~~~if $i=l$, $k=j$, $i\neq k$;}\\[1mm]
0,  \text{\quad\quad\quad~otherwise.}
\end{array}\right.
\end{equation*}
Moreover,
\begin{equation}\label{Definition of Fijkl}
\begin{split}
G^{i\overline{i},k\overline{k}} & = (1-\delta_{ik})(\sigma_{2}(\eta))^{-1}-(\sigma_{2}(\eta))^{-2}\sigma_{1}(\eta|i)\sigma_{1
}(\eta|k),\\
G^{i\overline{k},k\overline{i}} & = -(\sigma_{2}(\eta))^{-1}.
\end{split}
\end{equation}
Without  a confusion, we use $\sigma_{1}$, $\sigma_{2}$ and $\sigma_{1}(i)$ to denote $\sigma_{1}(\eta)$, $\sigma_{2}(\eta)$ and $\sigma_{1}(\eta|i)$, respectively.
The following inequalities are very useful.

\begin{lemma}
At $x_{0}$, we have
\begin{equation}\label{lower bound of sum Fii}
\sum_{i}G^{i\overline{i}}\geq \frac{2(n-1)}{n}\sigma_{2}^{-\frac{1}{2}},
\end{equation}
\begin{equation}\label{Inequality 3}
\eta_{1}\sigma_{1}(1)\geq \frac{2}{n}\sigma_{2},
\end{equation}
\begin{equation}\label{Inequality 2}
G^{i\ol{i}}\geq C\sum_{k}G^{k\ol{k}} \text{~for~} i\geq 2.
\end{equation}

\end{lemma}

\begin{proof}
(\ref{lower bound of sum Fii}) and (\ref{Inequality 3}) are direct consequences of Maclaurin's inequality. For  a  proof of (\ref{Inequality 2}), see \cite[Theorem 1]{LT94}.
\end{proof}

We define a second order operator on  $(M,\omega,J)$ by
\begin{equation*}
L(f)=G^{i\overline{j}}f_{i\overline{j}},
\end{equation*}
where $f\in C^{2}(M)$. It is clear that $L$ is the linearized operator of (\ref{General 2-nd Hessian equation}). Since $\chi+\ddbar\vp\in\Gamma_{2}(M)$, $L$ is a second order elliptic operator.  Here we use Einstein notation convention for convenience.

\subsection{An auxiliary function}
As  mentioned in Section 1, we follow the argument in \cite{CTW16} to obtain estimate (\ref{c2-varphi}). For any smooth function $\varphi$,
 we denote the eigenvalues of $\nabla^{2}\vp$ by $\lambda_{1}(\nabla^{2}\vp)\geq\lambda_{2}(\nabla^{2}\vp)\geq\cdots\geq\lambda_{2n}(\nabla^{2}\vp)$.  Since $(\chi+\ddbar\vp)\in \Gamma_{2}(M)\subset\Gamma_{1}(M)$,
\begin{equation*}
|\nabla^{2}\vp|_{g}\leq C\lambda_{1}(\nabla^{2}\vp)+C,
\end{equation*}
for a uniform constant $C$. Hence, it suffices to estimate $\lambda_{1}(\nabla^{2}\vp)$.   On the open set
 $M_+=\{x\in M~|~\lambda_{1}(\nabla^{2}\vp)>0\}$,  we consider the following quantity
\begin{equation*}
Q=\log\lambda_{1}(\nabla^{2}\vp)+h(|\partial\vp|_{g}^{2})+e^{-A\varphi},
\end{equation*}
 where $A$ is a constant to be determined. Without loss of generality, we may assume that
 $M_+$  is nonempty. Otherwise, we get upper bound of $\lambda_{1}(\nabla^{2}\vp)$ directly. Here
\begin{equation*}
h(s)=-\frac{1}{2}\log(1+\sup_{M}|\partial\varphi|^{2}_{g}-s), ~\forall~s\ge 0.
\end{equation*}
Then
\begin{equation}\label{Properties of h}
\frac{1}{2}\geq h^{'}\geq \frac{1}{2+2\sup_{M}|\partial\varphi|^{2}_{g}} \text{~and~} h''\geq 2(h')^{2}.
\end{equation}

We assume that $Q$ attains its maximum at $x_{0}$ on $M_+$. Near $x_{0}$, there exists a local unitary frame $\{e_{i}\}_{i=1}^{n}$ (with respect to $g$) such that at $x_{0}$, we have
\begin{equation*}
\text{$g_{i\overline{j}}=\delta_{ij}$, $\tilde{g}_{i\overline{j}}=\delta_{ij}\tilde{g}_{i\overline{j}}$ and $\tilde{g}_{1\overline{1}}\geq\tilde{g}_{2\overline{2}}\geq\cdots\geq\tilde{g}_{n\overline{n}}$.}
\end{equation*}
For convenience, we denote $\tilde{g}_{i\overline{i}}(x_{0})$ by $\eta_{i}$. On the other hand, since $(M,\omega,J)$ is almost Hermitian, we can find a normal coordinate system $(U,\{x^{\alpha}\}_{i=1}^{2n})$ around $x_{0}$ such that it holds at $x_{0}$,
\begin{equation}\label{real frame and complex frame}
e_{i}=\frac{1}{\sqrt{2}}(\partial_{2i-1}-\sqrt{-1}\partial_{2i}) \text{~for~} i=1,2,\cdots,n
\end{equation}
and
\begin{equation*}
\frac{\partial g_{\alpha\beta}}{\partial x^{\gamma}}=0 \text{~for~} \alpha,\beta,\gamma=1,2,\cdots,2n.
\end{equation*}
Let $V_{1},V_{2},\cdots,V_{n}$ be $g$-unit eigenvectors of $\nabla^{2}\vp$  corresponding to eigenvalues $\lambda_{1}(\nabla^{2}\vp),\lambda_{2}(\nabla^{2}\vp),\cdots,\lambda_{2n}(\nabla^{2}\vp)$ at $x_{0}$. We assume that $V_{\alpha}=V_{\alpha}^{\beta}\partial_{\beta}$ at $x_{0}$ and extend vector $V_{\alpha}$ to  vector fields  on $U$ by taking the components $V_{\alpha}^{\beta}$ to be constant.

When $\lambda_{1}(\nabla^{2}\vp)(x_{0})=\lambda_{2}(\nabla^{2}\vp)(x_{0})$, $\lambda_{1}(\nabla^{2}\vp)$ is not smooth near $x_{0}$. To avoid this non-smooth case, we apply a perturbation argument as in  \cite{CTW16,Sze15,STW15}).  We define an endomorphism $\Phi$ of $TM$  on $U$ by
\begin{equation}\label{Definition of Phi}
\begin{split}
\Phi & = \Phi_{\alpha}^{\beta}~\frac{\partial}{\partial x^{\alpha}}\otimes dx^{\beta}\\
     & = (g^{\alpha\gamma}\vp_{\gamma\beta}-g^{\alpha\gamma}B_{\gamma\beta})\frac{\partial}{\partial x^{\alpha}}\otimes dx^{\beta},
\end{split}
\end{equation}
where $B_{\gamma\beta}=\delta_{\gamma\beta}-V_{1}^{\gamma}V_{1}^{\beta}$.  Let  $\lambda_{1}(\Phi)\geq\lambda_{2}(\Phi)\geq\cdots\geq\lambda_{2n}(\Phi)$  be  the eigenvalues of $\Phi$.  Then  $V_{1},V_{2},\cdots,V_{2n}$ are still eigenvectors of $\Phi$, corresponding to eigenvalues $\lambda_{1}(\Phi),\lambda_{2}(\Phi),\cdots,\lambda_{2n}(\Phi)$ at $x_{0}$.   Moreover, $\lambda_{1}(\Phi)(x_{0})>\lambda_{2}(\Phi)(x_{0})$, which implies $\lambda_{1}(\Phi)$ is smooth near $x_{0}$.   On $U$, we replace $Q$ by  the following smooth quantity
\begin{equation*}
\hat{Q}=\log\lambda_{1}(\Phi)+h(|\partial\vp|_{g}^{2})+e^{-A\varphi}.
\end{equation*}
Since $\lambda_{1}(\nabla^{2}\vp)(x_{0})=\lambda_{1}(\Phi)(x_{0})$ and $\lambda_{1}(\nabla^{2}\vp)\geq\lambda_{1}(\Phi)$, $x_{0}$ is still the maximum point of $\hat{Q}$. For convenience,  we denote $\lambda_{\alpha}(\Phi)$ by $\lambda_{\alpha}$ for $\alpha=1,2,\cdots,2n$.

The following formulas give the first and second derivatives of $\lambda_{1}$ at $x_{0}$ (see e.g. \cite[Lemma 5.2]{CTW16}).

\begin{lemma}\label{First and second derivatives of lambda1}
\begin{equation}
\begin{split}
~&\lambda_{1}^{\alpha\beta}:=\frac{\partial \lambda_{1}}{\partial \Phi^{\alpha}_{\beta}}=V_{1}^{\alpha}V_{1}^{\beta},\\
~&\lambda_{1}^{\alpha\beta,\gamma\delta}:=\frac{\partial^{2}\lambda_{1}}{\partial \Phi^{\alpha}_{\beta}\partial\Phi^{\gamma}_{\delta}}=
\sum_{\mu>1}\frac{V_{1}^{\alpha}V_{\mu}^{\beta}V_{\mu}^{\gamma}V_{1}^{\delta}+V_{\mu}^{\alpha}V_{1}^{\beta}V_{1}^{\gamma}V_{\mu}^{\delta}}{\lambda_{1}-\lambda_{\mu}}.
\end{split}
\end{equation}
where $\alpha,\beta,\gamma,\delta=1,2,\cdots,2n$.
\end{lemma}

\section{Lower bound of $L(\hat{Q})$}

In this section  we  compute $L(\hat{Q})$ by using equation (\ref{General 2-nd Hessian equation}).
Since the right hand side $F$  of (\ref{General 2-nd Hessian equation})  depends on $\partial\vp$, a trouble  is that  a bad term $-C\lambda_{1}$ will appear when we differentiate (\ref{General 2-nd Hessian equation}) twice.  We  use the structure of the operator $\log\sigma_{2}$  to overcome it (see Lemma \ref{control C lambda1}).

Locally, $F(z,\partial\varphi,\varphi)$ can be regarded as a real-valued function on the set $\Gamma=U\times\mathbb{C}^{n}\times\mathbb{R}$. We denote points in $\Gamma$ typically by $\gamma=(z,p,r)$ where $z\in U$, $p=(p_{1},p_{2},\cdots,p_{n})\in\mathbb{C}^{n}$ and $r\in\mathbb{R}$.  For convenience, we use the following notations
\begin{equation*}
\begin{split}
F_{r} =\frac{\partial F}{\partial r},F_{p_{i}} & =\frac{\partial F}{\partial p_{i}},F_{\overline{p}_{i}}=\frac{\partial F}{\partial \overline{p}_{i}},\\
F_{i} =e_{i}(F(\cdot,p,r)),F_{\overline{i}} & =\overline{e}_{i}(F(\cdot,p,r)),F_{W}=W(F(\cdot,p,r)),
\end{split}
\end{equation*}
where $W$ is a vector field.
In the following, we always compute derivatives at the maximal point $x_0$ of $\hat Q$.
First we show

\begin{lemma}\label{L partial varphi}
\begin{equation}\label{l-nabla-varphi}
\begin{split}
L(|\partial \varphi|_{g}^{2}) & \geq \frac{1}{2}\sum_{k} G^{i\overline{i}}(|e_{i}e_{k}(\varphi)|^{2}+|e_{i}\overline{e}_{k}(\varphi)|^{2})
-C\sum_{i} G^{i\overline{i}}\\
                               &\quad+2\sum_{k,i}{\rm Re}\left(\vp_{k}(F_{p_{i}}\overline{e}_k{e}_{i}(\varphi)+F_{\overline{p}_{i}}\overline{e}_k\overline{e}_{i}(\varphi))\right).
\end{split}
\end{equation}
\end{lemma}

\begin{proof} By (\ref{General 2-nd Hessian equation}), we have
\begin{equation}\label{L partial varphi equation 1}
\log\sigma_{2}(\tilde{\omega}) =\log\left(
\begin{matrix}
n\\
2
\end{matrix}
\right)+F(z,\partial\varphi,\varphi),
\end{equation}
where $\tilde{\omega}=\chi+\ddbar\vp$. For any vector field $W$, differentiating (\ref{L partial varphi equation 1}) along $W$  at $x_{0}$,  we get
\begin{equation}
G^{i\overline{j}}W(\tilde{g}_{i\overline{j}})=W(F),
\end{equation}
which implies
\begin{equation}\label{differentiating once to equation}
\begin{split}
& \sum_{k}G^{i\overline{i}}(W e_{i}\overline{e}_{i}(\varphi)-W[e_{i},\overline{e}_{i}]^{(0,1)}(\varphi))\\
&= -G^{i\overline{i}}W(\chi_{i\overline{i}})+F_{W}+
    F_{r}W(\varphi)+F_{p_{i}}W{e}_{i}(\varphi)+F_{\overline{p}_{i}}W\overline{e}_{i}(\varphi).\\[2mm]
\end{split}
\end{equation}
By choosing $W=\ol{e}_{k}$, it follows
\begin{equation}\label{L partial varphi equation 2}
\begin{split}
& \sum_{k}G^{i\overline{i}}(\vp_{k}\overline{e}_{k}e_{i}\overline{e}_{i}(\varphi)-\vp_{k}\overline{e}_{k}[e_{i},\overline{e}_{i}]^{(0,1)}(\varphi))\\
& \geq   2\sum_{i,k}\left(\vp_{k}F_{p_{i}}\overline{e}_k{e}_{i}(\varphi)+\vp_{k}F_{\overline{p}_{i}}\overline{e}_k\overline{e}_{i}(\varphi)\right)
       -C\sum_{i}G^{i\overline{i}}.
\end{split}
\end{equation}

On the other hand,
\begin{equation}\label{L partial varphi equation 3}
\begin{split}
L(|\partial\vp|_{g}^{2})
 &= \sum_{k}G^{i\overline{i}}\left(e_{i}\overline{e}_{i}(\vp_{k}\vp_{\overline{k}})
                             -[e_{i},\overline{e}_{i}]^{(0,1)}(\vp_{k}\vp_{\overline{k}})\right)\\
& = \sum_{k}G^{i\overline{i}}(|e_{i}e_{k}(\vp)|^{2}+|e_{i}\overline{e}_{k}(\vp)|^{2})\\
&+\sum_{k}G^{i\overline{i}}\left(\vp_{k}e_{i}\overline{e}_{i}\overline{e}_{k}(\vp)
    -\vp_{k}[e_{i},\overline{e}_{i}]^{(0,1)}\overline{e}_{k}(\vp)\right)\\
&   +\sum_{k}G^{i\overline{i}}\left(\vp_{\overline{k}}e_{i}\overline{e}_{i}e_{k}(\vp)-\vp_{\overline{k}}[e_{i},\overline{e}_{i}]^{(0,1)}e_{k}(\vp)\right).
\end{split}
\end{equation}
Note
\begin{equation*}
\begin{split}
& \sum_{k}G^{i\overline{i}}\left(\vp_{k}e_{i}\overline{e}_{i}\overline{e}_{k}(\vp)-\vp_{k}[e_{i},\overline{e}_{i}]^{(0,1)}\overline{e}_{k}(\vp)\right)\\
&\geq \sum_{k}G^{i\overline{i}}\left(\vp_{k}\overline{e}_{k}e_{i}\overline{e}_{i}(\vp)-\vp_{k}\overline{e}_{k}[e_{i},\overline{e}_{i}]^{(0,1)}(\vp)\right)
       -C\sum_{i}G^{i\overline{i}}\\
& -C\sum_{k}G^{i\overline{i}}(|e_{i}e_{k}(\varphi)|+|e_{i}\overline{e}_{k}(\varphi)|).
    \end{split}
 \end{equation*}
By (\ref{L partial varphi equation 2}) and the Cauchy-Schwarz inequality,     it follows
  \begin{equation*}
\begin{split}
& ~\sum_{k}G^{i\overline{i}}\left(\vp_{k}e_{i}\overline{e}_{i}\overline{e}_{k}(\vp)-\vp_{k}[e_{i},\overline{e}_{i}]^{(0,1)}\overline{e}_{k}(\vp)\right)\\
&\geq ~ 2\sum_{i,k}\left(\vp_{k}F_{p_{i}}\overline{e}_k{e}_{i}(\varphi)+\vp_{k}F_{\overline{p}_{i}}\overline{e}_k\overline{e}_{i}(\varphi)\right)
       -C\sum_{i}G^{i\overline{i}}\\
& -\frac{1}{4}\sum_{k}G^{i\overline{i}}(|e_{i}e_{k}(\varphi)|^{2}+|e_{i}\overline{e}_{k}(\varphi)|^{2}).
\end{split}
\end{equation*}
Similarly,
\begin{equation*}
\begin{split}
& ~\sum_{k}G^{i\overline{i}}\left(\vp_{\overline{k}}e_{i}\overline{e}_{i}e_{k}(\vp)-\vp_{\overline{k}}[e_{i},\overline{e}_{i}]^{(0,1)}e_{k}(\vp)\right)\\
&\geq ~ 2\sum_{i,k}\left(\vp_{\overline{k}}F_{p_{i}}e_{k}{e}_{i}(\varphi)
       +\vp_{\overline{k}}F_{\overline{p}_{i}}e_{k}\overline{e}_{i}(\varphi)\right)-C\sum_{i}G^{i\overline{i}}\\
& -\frac{1}{4}\sum_{k}G^{i\overline{i}}(|e_{i}e_{k}(\varphi)|^{2}+|e_{i}\overline{e}_{k}(\varphi)|^{2}).
\end{split}
\end{equation*}
Substituting the above two inequalities into (\ref{L partial varphi equation 3}), we get (\ref{l-nabla-varphi}) immediately.
\end{proof}

Next, we compute  $L(\lambda_{1})$.

\begin{lemma}\label{the differential of eigenvalue}
\begin{equation*}
\begin{split}
L(\lambda_{1}) & \geq  2\sum_{\alpha>1}\frac{G^{i\overline{i}}|e_{i}(\varphi_{V_{\alpha}V_{1}})|^{2}}{(\lambda_{1}-\lambda_{\alpha})}
                       -{G^{i\overline{j},k\overline{l}}V_{1}(\tilde{g}_{i\overline{j}})V_{1}(\tilde{g}_{k\overline{l}})}\\
               & \quad -2G^{i\overline{i}}[V_{1},e_{i}]V_{1}\overline{e}_{i}(\varphi)-
                       2G^{i\overline{i}}[V_{1},\overline{e}_{i}]V_{1}{e}_{i}(\varphi)\\[2mm]
               & \quad -C\lambda_{1}\sum_{i}G^{i\overline{i}}-C\lambda_{1}^{2}+F_{p_{i}}V_{1}V_{1}e_{i}(\varphi)
                       +F_{\overline{p}_{i}}V_{1}V_{1}\overline{e}_{i}(\varphi).
\end{split}
\end{equation*}
\end{lemma}

\begin{proof}
The proof  is similar to one of \cite[Lemma 5.3]{CTW16}.  In fact, by
Lemma \ref{First and second derivatives of lambda1}, we have
\begin{equation}\label{the differential of eigenvalue equation 1}
\begin{split}
&L(\lambda_{1}) \\& =
G^{i\overline{i}}\lambda_{1}^{\alpha\beta,\gamma\delta}e_{i}(\Phi^{\gamma}_{\delta})
\overline{e}_{i}(\Phi^{\alpha}_{\beta})+G^{i\overline{i}}\lambda_{1}^{\alpha\beta}e_{i}\overline{e}_{i}
(\Phi_{\beta}^{\alpha})-G^{i\overline{i}}\lambda_{1}^{\alpha\beta}[e_{i},\overline{e}_{i}]^{(0,1)}(\Phi_{\beta}^{\alpha})\\
& = G^{i\overline{i}}\lambda_{1}^{\alpha\beta,\gamma\delta}e_{i}(\varphi_{\alpha\beta})
\overline{e}_{i}(\varphi_{\alpha\beta})
+G^{i\overline{i}}\lambda_{1}^{\alpha\beta}e_{i}\overline{e}_{i}(\varphi_{\alpha\beta})
+G^{i\overline{i}}\lambda_{1}^{\alpha\beta}\varphi_{\gamma\beta}e_{i}\overline{e}_{i}(g^{\alpha\gamma})\\
& \quad -G^{i\overline{i}}\lambda_{1}^{\alpha\beta}B_{\gamma\beta}e_{i}\overline{e}_{i}(g^{\alpha\gamma})
-G^{i\overline{i}}\lambda_{1}^{\alpha\beta}[e_{i},\overline{e}_{i}]^{(0,1)}(\varphi_{\alpha\beta})\\
& \geq
2\sum_{\alpha>1}G^{i\overline{i}}\frac{|e_{i}(\varphi_{V_{\alpha}V_{1}})|^{2}}{\lambda_{1}-\lambda_{\alpha}}
 -C\lambda_{1}\sum_{i}G^{i\overline{i}}\\
&+G^{i\overline{i}}e_{i}\overline{e}_{i}(\varphi_{V_{1}V_{1}})
-G^{i\overline{i}}[e_{i},\overline{e}_{i}]^{(0,1)}(\varphi_{V_{1}V_{1}}).
\end{split}
\end{equation}
We need to deal with last two terms in (\ref{the differential of eigenvalue equation 1}). Note $|e_{i}\overline{e}_{i}(\nabla_{V_1}V_1)(\vp)-(\nabla_{V_1}V_1)e_{i}\overline{e}_{i}(\vp)|\leq C\lambda_{1}$.
 Then by  \eqref{differentiating once to equation}, we have
\begin{equation*}
\left|G^{i\overline{i}}e_{i}\overline{e}_{i}(\nabla_{V_1}V_1)(\vp)\right|\leq C\lambda_{1} \sum_{i} G^{i\overline{i}}+C\lambda_{1}.
\end{equation*}
It follows
\begin{equation*}
\begin{split}
&  G^{i\overline{i}}e_{i}\overline{e}_{i}(\vp_{V_{1}V_{1}})-G^{i\overline{i}}[e_{i},\overline{e}_{i}]^{(0,1)}(\vp_{V_{1}V_{1}})\\
& =     G^{i\overline{i}}e_{i}\overline{e}_{i}V_{1}V_{1}(\vp)-G^{i\overline{i}}e_{i}\overline{e}_{i}(\nabla_{V_{1}}V_{1})(\vp)
          -G^{i\overline{i}}[e_{i},\overline{e}_{i}]^{(0,1)}V_{1}V_{1}(\vp)\\
       &  +G^{i\overline{i}}[e_{i},\overline{e}_{i}]^{(0,1)}(\nabla_{V_{1}}V_{1})(\vp)\\
& \geq   G^{i\overline{i}}e_{i}\overline{e}_{i}V_{1}V_{1}(\vp)-G^{i\overline{i}}[e_{i},\overline{e}_{i}]^{(0,1)}V_{1}V_{1}(\vp)
          -C\lambda_{1}\sum_{i}G^{i\overline{i}}-C\lambda_{1}.
\end{split}
\end{equation*}
By using  the Lie bracket for vector fields, we further get
\begin{equation*}
\begin{split}
& G^{i\overline{i}}e_{i}\overline{e}_{i}V_{1}V_{1}(\varphi)-G^{i\overline{i}}[e_{i},\overline{e}_{i}]^{(0,1)}V_{1}V_{1}(\varphi) \\
 \geq ~~& G^{i\overline{i}}\left(V_{1}e_{i}\overline{e}_{i}V_{1}(\vp)+[e_{i},V_{1}]\overline{e}_{i}V_{1}(\vp)
          -[V_{1},\overline{e}_{i}]e_{i}V_{1}(\vp)-V_{1}V_{1}[e_{i},\overline{e}_{i}]^{(0,1)}(\vp)\right)\\
        & -C\lambda_{1}\sum_{i}G^{i\overline{i}}\\
 \geq ~~&G^{i\overline{i}}V_{1}V_{1}\left(e_{i}\overline{e}_{i}(\varphi)-[e_{i},\overline{e}_{i}]^{(0,1)}(\varphi)\right)
          -2G^{i\overline{i}}[V_{1},e_{i}]V_{1}\overline{e}_{i}(\vp)\\
        & -2G^{i\overline{i}}[V_{1},\overline{e}_{i}]V_{1}e_{i}(\vp)-C\lambda_{1}\sum_{i}G^{i\overline{i}}.
\end{split}
\end{equation*}
Thus
\begin{equation}\label{the differential of eigenvalue equation 4}
\begin{split}
&G^{i\overline{i}}e_{i}\overline{e}_{i}(\varphi_{V_{1}V_{1}})-G^{i\overline{i}}[e_{i},\overline{e}_{i}]^{(0,1)}(\varphi_{V_{1}V_{1}})\\
 \geq~~& G^{i\overline{i}}V_{1}V_{1}(\tilde{g}_{i\overline{i}})-2G^{i\overline{i}}[V_{1},e_{i}]V_{1}\overline{e}_{i}(\vp)
-2G^{i\overline{i}}[V_{1},\overline{e}_{i}]V_{1}e_{i}(\vp)\\
&-C\lambda_{1}\sum_{i}G^{i\overline{i}}-C\lambda_{1}.
\end{split}
\end{equation}

On the other hand, differentiating (\ref{L partial varphi equation 1}) along $V_1$  twice  at $x_{0}$,
we have
\begin{equation*}
G^{i\overline{i}}V_{1}V_{1}(\tilde{g}_{i\overline{i}})+G^{i\overline{j},k\overline{l}}V_{1}(\tilde{g}_{i\overline{j}})V_{1}(\tilde{g}_{k\overline{l}})=F_{p_{i}}V_{1}V_{1}e_{i}(\varphi)
+F_{\overline{p}_{i}}V_{1}V_{1}\overline{e}_{i}(\varphi)+E,
\end{equation*}
where $E$ denotes a term satisfying $|E|\leq C\lambda_{1}^{2}$ for a uniform constant $C$.  Thus by  (\ref{the differential of eigenvalue equation 4}), we get
\begin{equation*}
\begin{split}
& G^{i\overline{i}}e_{i}\overline{e}_{i}(\varphi_{V_{1}V_{1}})-G^{i\overline{i}}[e_{i},\overline{e}_{i}]^{(0,1)}(\varphi_{V_{1}V_{1}})\\
\geq ~& -{G^{i\overline{j},k\overline{l}}V_{1}(\tilde{g}_{i\overline{j}})V_{1}(\tilde{g}_{k\overline{l}})}
        -2G^{i\overline{i}}[V_{1},e_{i}]V_{1}\overline{e}_{i}(\varphi)-2G^{i\overline{i}}[V_{1},\overline{e}_{i}]V_{1}{e}_{i}(\varphi)\\
     & -C\lambda_{1}\sum_{i}G^{i\overline{i}}-C\lambda_{1}^{2}+F_{p_{i}}V_{1}V_{1}e_{i}(\varphi)+F_{\overline{p}_{i}}V_{1}V_{1}\overline{e}_{i}(\varphi).
\end{split}
\end{equation*}
Substituting  the above inequality into (\ref{the differential of eigenvalue equation 1}), we prove the lemma.
\end{proof}

By Lemma \ref{L partial varphi} and Lemma \ref{the differential of eigenvalue}, we get
\begin{equation}\label{LQ Lemma equation 4}
\begin{split}
&L(\log\lambda_{1}(\Phi)+h(|\partial\vp|_{g}^{2}))\\
 &= \frac{L(\lambda_{1})}{\lambda_{1}}+h'L(|\partial \varphi|^{2}_{g})
 -\frac{G^{i\overline{i}}|e_{i}(\vp_{V_{1}V_{1}})|^{2}}{\lambda_{1}^{2}}+
         h''G^{i\overline{i}}|e_{i}|\partial\varphi|^{2}_{g}|^{2} \\
&\ge   2\sum_{\alpha>1}\frac{G^{i\overline{i}}|e_{i}(\varphi_{V_{\alpha}V_{1}})|^{2}}{\lambda_{1}(\lambda_{1}-\lambda_{\alpha})}
-\frac{G^{i\ol{j},k\ol{l}}V_{1}(\tilde{g}_{i\ol{j}})V_{1}(\tilde{g}_{k\ol{l}})}{\lambda_{1}}\\
  &+\frac{h'}{2}\sum_{k}G^{i\overline{i}}(|e_{i}e_{k}(\varphi)|^{2}+|e_{i}\overline{e}_{k}(\varphi)|^{2})+
         h''G^{i\overline{i}}|e_{i}|\partial\varphi|^{2}_{g}|^{2}\\
  &-\frac{G^{i\overline{i}}|e_{i}(\vp_{V_{1}V_{1}})|^{2}}{\lambda_{1}^{2}}-C\sum_{i}G^{i\overline{i}}\\
& -\frac{2G^{i\overline{i}}[V_{1},e_{i}]V_{1}\overline{e}_{i}(\varphi)+  2G^{i\overline{i}}[V_{1},\overline{e}_{i}]V_{1}{e}_{i}(\varphi)}{\lambda_{1}}\\
&+ [F_{p_{i}}(\frac{ V_{1}V_{1}e_{i}(\varphi)}{\lambda_{1}}+h'(\vp_{\overline{k}}
 e_{k}e_{i}(\varphi)+\vp_{k}\overline{e}_{k}e_{i}(\vp))]\\
&  +[F_{\overline{p}_{i}} (\frac{V_{1}V_{1}\overline{e}_{i}(\varphi)}{\lambda_{1}}+h'(\vp_{\overline{k}} e_{k}\overline{e}_{i}(\varphi)+\vp_{k}\overline{e}_{k}\overline{e}_{i}(\vp))]\\
&  -C\lambda_{1}.
\end{split}
\end{equation}
We need to deal with last fourth terms in (\ref{LQ Lemma equation 4}) where   three  parts are about the 3th-derivative of $\varphi$ and one is an eigenvalue function. The term
$$-\frac{2G^{i\overline{i}}[V_{1},e_{i}]V_{1}\overline{e}_{i}(\varphi)+  2G^{i\overline{i}}[V_{1},\overline{e}_{i}]V_{1}{e}_{i}(\varphi)}{\lambda_{1}}$$
 can be handled as
\begin{equation}\label{LQ Lemma equation 2}
\begin{split}
& \frac{2G^{i\overline{i}}[V_{1},e_{i}]V_{1}\overline{e}_{i}(\varphi)+  2G^{i\overline{i}}[V_{1},\overline{e}_{i}]V_{1}{e}_{i}(\varphi)}{\lambda_{1}}\\
\leq &~ \ve \frac{G^{i\overline{i}}|e_{i}(\varphi_{V_1V_{1}})|^{2}}{\lambda_{1}^{2}}+\ve\sum_{\alpha>1}\frac{G^{i\overline{i}}|e_{i}(\varphi_{V_\alpha V_{1}})|^{2}}
{\lambda_{1}(\lambda_{1}-\lambda_{\alpha})}+\frac{C}{\ve}\sum_{i}G^{i\overline{i}}.
\end{split}
\end{equation}
Here $\ve\in(0,\frac{1}{2}]$ is a constant to be determined later. We refer the reader to a similar argument in \cite[Lemma 5.4] {CTW16}.

To control the term $(\frac{ V_{1}V_{1}e_{i}(\varphi)}{\lambda_{1}}+h'(\vp_{\overline{k}}e_{k}e_{i}(\varphi)+\vp_{k}\overline{e}_{k}e_{i}(\vp))$ in (\ref{LQ Lemma equation 4}). We use the fact $d\hat Q(x_0)=0$.   In fact,
\begin{equation}\label{first-derivative}
\begin{split}
\frac{e_{i}(\varphi_{V_{1}V_{1}})}{\lambda_{1}}
&   =   Ae^{-A\vp}e_{i}(\varphi)-h'e_{i}(|\partial\varphi|_{g}^{2})\\
&   =   Ae^{-A\vp}e_{i}(\varphi)-h'\left(\vp_{\overline{k}} e_{i}e_{k}(\varphi)+\vp_{k}e_{i}\ol{e}_{k}(\vp)\right).
\end{split}
\end{equation}
Note
\begin{equation*}
\left|V_{1}V_{1}e_{i}(\vp) - e_{i}(\vp_{V_{1}V_{1}})\right|\leq C\lambda_{1}.
\end{equation*}
Thus
\begin{equation}\label{LQ Lemma vanish equation 1}
\begin{split}
&\left |   F_{p_{i}} \left(    \frac{V_{1}V_{1}e_{i}(\varphi)}{\lambda_{1}}+h'  (\vp_{\overline{k}}
  e_{k}e_{i}(\varphi)+\vp_{k}\overline{e}_{k}e_{i}(\vp)  )  \right) \right|\\
&\leq |F_{p_{i}}|\cdot\left|\frac{V_{1}V_{1}e_{i}(\varphi)}{\lambda_{1}}-\frac{e_{i}(\vp_{V_{1}V_{1}})}{\lambda_{1}}+Ae^{-A\vp}e_{i}(\vp)\right|\\
& \leq  CAe^{-A\vp}.
\end{split}
\end{equation}
Similarly, we have
\begin{equation}\label{LQ Lemma vanish equation 2}
 \left|  F_{\overline{p}_i} \left(  \frac{V_{1}V_{1}\overline{e}_{i}(\varphi)}{\lambda_{1}}+  h'  (\vp_{\overline{k}} e_{k}\overline{e}_{i}(\varphi)+\vp_{k}\overline{e}_{k}\overline{e}_{i}(\vp)  )  \right)\right |\leq CAe^{-A\varphi}.
\end{equation}

The following lemma gives a control to $\lambda_{1}$  for the solution $\varphi$ in (\ref{General 2-nd Hessian equation}).

\begin{lemma}\label{control C lambda1}
\begin{equation*}
C\lambda_{1} \leq \frac{h'}{4}\sum_{k}G^{i\overline{i}}(|e_{i}e_{k}(\varphi)|^{2}+|e_{i}\overline{e}_{k}(\varphi)|^{2})+C\sum_{i}G^{i\overline{i}}.
\end{equation*}
\end{lemma}

\begin{proof}
At $x_{0}$, by (\ref{the first order derivative of F}), we have
\begin{equation*}
G^{i\overline{j}}=\frac{\sigma_{1}(i)}{\sigma_{2}}\delta_{ij},
\end{equation*}
where $\eta_{i}=\tilde{g}_{i\overline{i}}$ and $\sigma_{1}(i)=\sum_{k\neq i}\eta_{k}$. It is clear that
\begin{equation*}
\begin{split}
\sigma_{1}(1)\sigma_{1}(\eta) & = (\sigma_{1}(1))^{2}+\eta_{1}\sigma_{1}(1)\\
& =\sum_{i\geq2}\eta_{i}^{2}+2\sum_{i>j\geq2}\eta_{i}\eta_{j}+\sum_{i\geq2}\eta_{1}\eta_{i}\\
& =\sum_{i\geq2}\eta_{i}^{2}+\sum_{i>j\geq2}\eta_{i}\eta_{j}+\sigma_{2}\\
& \geq \sigma_{2},
\end{split}
\end{equation*}
which implies
\begin{equation*}
\frac{1}{G^{i\overline{i}}}\leq\frac{1}{G^{1\overline{1}}}=\frac{\sigma_{2}}{\sigma_{1}(1)}\leq \sigma_{1}=\frac{\sigma_{2}}{n-1}\sum_{k}G^{k\overline{k}}\leq C\sum_{k}G^{k\overline{k}},~i=1,2,\cdots,n.
\end{equation*}
Combining this with the Cauchy-Schwarz inequality and (\ref{Properties of h}), we have
\begin{equation*}
\begin{split}
C\lambda_{1} & \leq \frac{h'}{4}G^{1\overline{1}}\lambda_{1}^{2}+\frac{C}{h'G^{1\overline{1}}}\\
& \leq \frac{h'}{4}\sum_{k}G^{i\overline{i}}(|e_{i}e_{k}(\varphi)|^{2}+|e_{i}\overline{e}_{k}(\varphi)|^{2})+C\sum_{i}G^{i\overline{i}},
\end{split}
\end{equation*}
as required.
\end{proof}

Substituting the above relations  into (\ref{LQ Lemma equation 4}), we get the main estimate in this section.

\begin{proposition}\label{LQ Lemma}
Let $\vp$ be the solution of (\ref{General 2-nd Hessian equation}). Then at $x_{0}$, there exists a uniform constant $C$ such that for any $\ve \in (0, \frac{1}{2}]$, it holds
\begin{equation}\label{main-estimate-1}
\begin{split}
0 & \geq(2-\ve)\sum_{\alpha>1}\frac{G^{i\overline{i}}|e_{i}(\varphi_{V_{\alpha}V_{1}})|^{2}}{\lambda_{1}(\lambda_{1}-\lambda_{\alpha})}
-\frac{G^{i\ol{j},k\ol{l}}V_{1}(\tilde{g}_{i\ol{j}})V_{1}(\tilde{g}_{k\ol{l}})}{\lambda_{1}}\\
&-(1+\ve)\frac{G^{i\overline{i}}|{e_{i}(\varphi_{V_{1}V_{1}})}|^{2}}{\lambda_{1}^{2}}\\
          &
             +\frac{h'}{4}\sum_{k}G^{i\overline{i}}(|e_{i}e_{k}(\varphi)|^{2}+|e_{i}\overline {e}_{k}(\varphi)|^{2})+h''G^{i\overline{i}}|e_{i}(|\partial\varphi|_{g}^{2})|^{2}\\
          & +\left(\ve_{0} Ae^{-A\varphi} -\frac{C}{\ve}\right)\sum_{i}G^{i\overline{i}}
+A^{2}e^{-A\varphi}G^{i\overline{i}}|\varphi_{i}|^{2}-CA e^{-A\varphi}.
\end{split}
\end{equation}
\end{proposition}

\begin{proof}

At $x_{0}$, we have
\begin{align}
0& \geq L(\hat{Q})\notag\\
&=L(\log\lambda_{1}+h(|\partial\vp|_{g}^{2})) -Ae^{-A\varphi}L(\varphi)+A^{2}e^{-A\varphi}G^{i\overline{i}}|e_{i}(\varphi)|^{2}.\notag
\end{align}
Note
\begin{equation*}
L(\vp)=G^{i\overline{i}}(\tilde{g}_{i\overline{i}}-\chi_{i\overline{i}})=2-G^{i\overline{i}}\chi_{i\overline{i}}\leq 2-\ve_{0}\sum_{i}G^{i\overline{i}}.
\end{equation*}
Thus by (\ref{LQ Lemma equation 4}) together with estimates (\ref{LQ Lemma equation 2}), (\ref{LQ Lemma vanish equation 1}), (\ref{LQ Lemma vanish equation 2})
and Lemma \ref{control C lambda1}, one get (\ref{main-estimate-1}) immediately.

\end{proof}

By concavity of $\log\sigma_{2}$ and (\ref{Definition of Fijkl}), we see that $-G^{k\overline{l},l\overline{k}}>0$ and $(-G^{i\overline{i},k\overline{k}})$ is a non-negative definite matrix. Hence, the "good" positive terms at the right hand of (\ref{main-estimate-1}) is
\begin{equation*}
I=(2-\ve)\sum_{\alpha>1}\frac{G^{i\overline{i}}|e_{i}(\varphi_{V_{\alpha}V_{1}})|^{2}}{\lambda_{1}(\lambda_{1}-\lambda_{\alpha})}
             -\frac{G^{k\overline{l},l\overline{k}}|V_{1}(\tilde{g}_{k\overline{l}})|^{2}}{\lambda_{1}}
             -\frac{G^{i\overline{i},k\overline{k}}V_{1}(\tilde{g}_{k\overline{k}})V_{1}(\tilde{g}_{i\overline{i}})}{\lambda_{1}}.
\end{equation*}
In next section, we will use this "good" positive terms to control the "bad"  term in (\ref{main-estimate-1}),
\begin{equation*}
II=(1+\ve)\frac{G^{i\overline{i}}|{e_{i}(\varphi_{V_{1}V_{1}})}|^{2}}{\lambda_{1}^{2}}.
\end{equation*}

As an application of Proposition \ref{LQ Lemma},  we  get  the following partial estimate of real Hessian $\nabla^2\varphi$.
\begin{corollary}\label{Lemma 2} There exists a uniform constant $C_{A}$  depending on $A$ such that
\begin{equation}\label{Lemma 2 equation 1}
\sum_{i=2}^{n}\sum_{k=1}^{n}(|e_{i}e_{k}(\vp)|^{2}+|e_{i}\overline{e}_{k}(\vp)|^{2})\leq C_{A}, \quad \sum_{i=2}^{n}|\eta_{i}|\leq C_{A}
\end{equation}
and
\begin{equation}\label{Lemma 2 equation 2}
\lambda_{1}\leq C_{A}\eta_{1}+C,
\end{equation}
where $\eta_{i}=\tilde{g}_{i\overline{i}}=\chi_{i\overline{i}}+\vp_{i\overline{i}}$ for $i=1,2,\cdots,n$.
\end{corollary}

\begin{proof}
By (\ref{first-derivative}), we have
\begin{equation*}
\begin{split}
-\frac{3}{2}\frac{G^{i\overline{i}}|e_{i}(\varphi_{V_{1}V_{1}})|^{2}}{\lambda_{1}^{2}}
&   =   -\frac{3}{2}G^{i\overline{i}}|Ae^{-A\vp}\varphi_{i}-h'e_{i}(|\partial\varphi|_{g}^{2})|^{2}\\
& \geq  -C_{A}\sum_{i}G^{i\overline{i}}-2(h')^{2}G^{i\overline{i}}|e_{i}(|\partial\varphi|_{g}^{2})|^{2}.
\end{split}
\end{equation*}
 Recall that the matrix $(-G^{i\overline{i},k\overline{k}})$ is non-negative and $-G^{k\overline{l},l\overline{k}}>0$. Then
\begin{equation*}
(2-\ve)\sum_{\alpha>1}\frac{G^{i\overline{i}}|e_{i}(\varphi_{V_{\alpha}V_{1}})|^{2}}{\lambda_{1}(\lambda_{1}-\lambda_{\alpha})}
             -\frac{G^{k\overline{l},l\overline{k}}|V_{1}(\tilde{g}_{k\overline{l}})|^{2}}{\lambda_{1}}
             -\frac{G^{i\overline{i},k\overline{k}}V_{1}(\tilde{g}_{k\overline{k}})V_{1}(\tilde{g}_{i\overline{i}})}{\lambda_{1}}\geq 0.
\end{equation*}
Thus by choosing $\ve=\frac{1}{2}$  in  (\ref{main-estimate-1}), we obtain
\begin{equation*}
\begin{split}
0 \geq ~~& \frac{h'}{4}\sum_{k}G^{i\overline{{i}}}(|e_{i}e_{k}(\varphi)|^{2}
+|e_{i}\overline{e}_{k}(\varphi)|^{2})+h''G^{i\overline{i}}|e_{i}(|\partial\varphi|_{g}^{2})|^{2}\\
&-2(h')^{2}G^{i\overline{i}}|e_{i}(|\partial\varphi|_{g}^{2})|^{2}-C_{A}\sum_{k}G^{k\overline{k}}-C_{A}.
\end{split}
\end{equation*}
By (\ref{lower bound of sum Fii}) and (\ref{Properties of h}), it  follows
\begin{equation}\label{Lemma 2 equation 3}
0 \geq \sum_{k}G^{i\overline{{i}}}(|e_{i}e_{k}(\varphi)|^{2}+|e_{i}\overline {e}_{k}(\varphi)|^{2})-C_{A}\sum_{k}G^{k\overline{k}}.
\end{equation}
Combining this with (\ref{Inequality 2}), we obtain
\begin{equation*}
\sum_{i=2}^{n}\sum_{k=1}^{n}(|e_{i}e_{k}(\vp)|^{2}+|e_{i}\overline{e}_{k}(\vp)|^{2})\leq C_{A}.
\end{equation*}
In particular, for $i\geq2$, it is clear that
\begin{equation*}
\eta_{i}=\chi_{i\overline{i}}+\vp_{i\overline{i}}=\chi_{i\overline{i}}+e_{i}\overline{e}_{i}(\vp)-[e_{i},\overline{e}_{i}]^{(0,1)}(\vp)\leq C_{A}.
\end{equation*}
Hence (\ref{Lemma 2 equation 1}) is true.

Next,  we prove (\ref{Lemma 2 equation 2}).  By (\ref{the first order derivative of F}) and (\ref{Inequality 3}), we see
\begin{equation*}
G^{n\overline{n}}\geq\cdots\geq G^{1\overline{1}}\geq \frac{1}{C\eta_{1}}.
\end{equation*}
Combining  this with (\ref{Lemma 2 equation 3}),   we have
\begin{equation*}
\begin{split}
\lambda_{1}^{2} &   =  \left(V_{1}V_{1}(\vp)-(\nabla_{V_{1}}V_{1})(\vp)\right)^{2}\\[1mm]
                & \leq \sum_{i,k}(|e_{i}e_{k}(\varphi)|^{2}+|e_{i}\overline {e}_{k}(\varphi)|^{2})+C\\
                & \leq C\eta_{1}\sum_{i,k}G^{i\overline{i}}(|e_{i}e_{k}(\varphi)|^{2}+|e_{i}\overline {e}_{k}(\varphi)|^{2})+C\\
                & \leq C_{A}\eta_{1}\sum_{k}G^{k\overline{k}}+C\\
                &   =  C_{A}\eta_{1}\sum_{k}\frac{\sigma_{1}(k)}{\sigma_{2}}+C\\
                & \leq C_{A}\eta_{1}^{2}+C,
\end{split}
\end{equation*}
where we used $\eta_{1}\geq\eta_{2}\geq\cdots\geq\eta_{n}$ in  the last inequality.  Thus  (\ref{Lemma 2 equation 2}) is true.
\end{proof}

Corollary  \ref{Lemma 2} will be used in next section.

\section{Estimate of  $II$}

We decompose  $II$ into three parts as follows,
\begin{equation*}
\begin{split}
&(1+\ve)\frac{G^{1\overline{1}}|{e_{1}(\varphi_{V_{1}V_{1}})}|^{2}}{\lambda_{1}^{2}}
+3\ve\sum_{i\geq 2}\frac{G^{i\overline{i}}|{e_{i}(\varphi_{V_{1}V_{1}})}|^{2}}{\lambda_{1}^{2}}
+(1-2\ve)\sum_{i\geq 2}\frac{G^{i\overline{i}}|{e_{i}(\varphi_{V_{1}V_{1}})}|^{2}}{\lambda_{1}^{2}}\\
&=:II_{1}+II_{2}+II_{3}.\\
\end{split}
\end{equation*}
In the following,  we always use  $C_{A}$ to denote a uniform constant depending on $A$. Without loss of generality, we may assume that $\lambda_{1}\geq\frac{C_{A}}{\ve}$.  We first estimate $II_{1}$ and $II_{2}$.

\begin{lemma}\label{II1 and II2}
\begin{equation*}
II_{1}\leq C_{A}+2(h')^{2}G^{1\overline{1}}|e_{1}(|\partial\vp|_{g}^{2})|^{2}
\end{equation*}
and
\begin{equation*}
II_{2} \leq 12\ve A^{2}e^{-2A\vp}\sum_{i\geq2}G^{i\overline{i}}|e_{i}(\vp)|^{2}
        +2(h')^{2}\sum_{i\geq2}G^{i\overline{i}}|e_{i}(|\partial\varphi|_{g}^{2})|^{2}.
\end{equation*}
\end{lemma}

\begin{proof}
Using (\ref{first-derivative}), we have
\begin{equation*}
II_{1} =\frac{G^{1\overline{1}}|e_{1}(\varphi_{V_{1}V_{1}})|^{2}}{\lambda_{1}^{2}}
=G^{1\overline{1}}|Ae^{-A\vp}e_{1}(\varphi)-h'e_{1}(|\partial\varphi|_{g}^{2})|^{2}.
\end{equation*}
Since  $G^{1\overline{1}}=\frac{\sigma_{1}(1)}{\sigma_2}\leq C$  by Corollary \ref{Lemma 2},  we get
\begin{equation*}
II_{1}
\leq  C_{A}+2(h')^{2}G^{1\overline{1}}|e_{1}(|\partial\varphi|_{g}^{2})|^{2}.
\end{equation*}
Similarly,
\begin{equation*}
\begin{split}
II_{2} & = 3\ve \sum_{i\geq2}\frac{G^{i\overline{i}}|e_{i}(\varphi_{V_{1}V_{1}})|^{2}}{\lambda_{1}^{2}}\\
&   =   3\ve G^{i\overline{i}}|Ae^{-A\vp}e_{i}(\varphi)-h'e_{i}(|\partial\varphi|_{g}^{2})|^{2}\\[2mm]
& \leq  12\ve A^{2}e^{-2A\vp}\sum_{i\geq2}G^{i\overline{i}}|e_{i}(\vp)|^{2}
        +4\ve(h')^{2}\sum_{i\geq2}G^{i\overline{i}}|e_{i}(|\partial\varphi|_{g}^{2})|^{2}\\
& \leq  12\ve A^{2}e^{-2A\vp}\sum_{i\geq2}G^{i\overline{i}}|e_{i}(\vp)|^{2}
        +2(h')^{2}\sum_{i\geq2}G^{i\overline{i}}|e_{i}(|\partial\varphi|_{g}^{2})|^{2}.
\end{split}
\end{equation*}
Here we used $0<\ve\leq\frac{1}{2}$ in the last inequality .
\end{proof}

In order to estimate $II_{3}$,  we  need  several lemmas below.  Let
\begin{equation*}
\tilde{e}=\frac{1}{\sqrt{2}}(V_{1}-\sqrt{-1}JV_{1}).
\end{equation*}
be  $(1,0)$-tpye vector field in the coordinate system $(U,\{x^{\alpha}\}_{\alpha=1}^{2n})$.
Since $\tilde{e}$ is $g$-unit, we can write $\tilde{e} $ at $x_0$ as
\begin{equation*}
\tilde{e}=\sum_{q} \nu_{q}e_{q} \text{~and~} \sum_{q=1}^{n}|\nu_q|^{2}=1,
\end{equation*}
for complex number $\nu_{1},\nu_{2},\cdots,\nu_{n}$. There are also  numbers $\mu_{\alpha}$ ($\alpha >1$) with $\sum_{\alpha>1}\mu_{\alpha}^{2}=1$ such that
\begin{equation*}
JV_{1}=\sum_{\alpha>1}\mu_{\alpha}V_{\alpha}.
\end{equation*}
 Then we have
\begin{equation}\label{Lemma 4 equation 1}
e_{i}(\varphi_{V_{1}V_{1}})=\sqrt{2}\sum_{q} \overline{\nu_{q}}V_{1}(\tilde{g}_{i\overline{q}})-
\sqrt{-1}\sum_{\alpha>1}\mu_{\alpha}e_{i}(\varphi_{V_{1}V_{\alpha}})+E,
\end{equation}
where $E$ denotes a term satisfying $|E|\leq C \lambda_{1}$.
 A similar computation of (\ref{Lemma 4 equation 1}) can be found in \cite[(5.31)]{CTW16}.

\begin{lemma}\label{Lemma 3}
\begin{equation*}
|\nu_{q}|\leq \frac{C_{A}}{\lambda_{1}}~ \text{~for $q\geq2$}.
\end{equation*}
\end{lemma}

\begin{proof}
By (\ref{Lemma 2 equation 1}), we have
\begin{equation*}
\sum_{i=2}^{n}\sum_{k=1}^{n}(|e_{i}e_{k}(\vp)|^{2}+|e_{i}\overline{e}_{k}(\vp)|^{2})\leq C_{A}.
\end{equation*}
Combining this with (\ref{real frame and complex frame}), we obtain
\begin{equation}\label{Lemma 3 equation 2}
\sum_{\alpha\geq 3}^{2n}\sum_{\beta\geq 1}^{2n}|\nabla_{\alpha\beta}^{2}\vp|\leq C_{A}.
\end{equation}
This means
\begin{equation*}
|\Phi_{\alpha}^{\beta}|\leq C_{A}~\text{~for $3\leq\alpha\leq2n$, $1\leq\beta\leq2n$}.
\end{equation*}
Recalling that $V_{1}$ is the eigenvector of $\Phi$ corresponding to $\lambda_{1}$, we have
\begin{equation*}
|V_{1}^{\alpha}|=\left|\frac{1}{\lambda_{1}}\sum_{\beta=1}^{2n}\Phi_{\beta}^{\alpha} V_{1}^{\beta}\right|\leq \frac{C}{\lambda_{1}}~ \text{~for $3\leq\alpha\leq 2n$}.
\end{equation*}
 Thus for any $q\geq2$,  we get
\begin{equation*}
|\nu_{q}|=|V_{1}^{2q-1}|+|V_{1}^{2q}|\leq\frac{C_{A}}{\lambda_{1}}.
\end{equation*}

\end{proof}

By Corollary \ref{Lemma 2} and Lemma \ref{Lemma 3}, we get an upper bound of $G^{i\overline{i}}$ for $i\geq 2$.

\begin{lemma}\label{Lemma 9}
For $i\geq2$, at $x_{0}$, if $\lambda_{1}\geq\frac{C_{A}}{\ve}$, we have
\begin{equation*}
(1-\ve)G^{i\overline{i}}\leq \frac{1}{2\sigma_{2}}\left(\lambda_{1}+\sum_{\alpha>1}\lambda_{\alpha}\mu_{\alpha}^{2}\right).
\end{equation*}
\end{lemma}

\begin{proof}
By the definition of $\tilde{e}$,  we see
$$\tilde{g}(\tilde{e},\overline{\tilde{e}})  = \sum_{q}|\nu_{q}|^{2}\eta_{q} = |\nu_{1}|^{2}\eta_{1} +\sum_{q=2}^{n}|\nu_{q}|^{2}\eta_{q}.$$
   By  Corollary \ref{Lemma 2} and Lemma \ref{Lemma 3}, it follows
\begin{equation*}
\tilde{g}(\tilde{e},\overline{\tilde{e}})
\geq  \left(1-\frac{C_{A}}{\lambda_{1}^{2}}\right)\eta_{1}-\frac{C_{A}}{\lambda_{1}^{2}}.
\end{equation*}
On the other hand,
\begin{equation*}
\begin{split}
&\tilde{g}(\tilde{e},\overline{\tilde{e}})\\
&  =  g(\tilde{e},\overline{\tilde{e}})+\tilde{e}\overline{\tilde{e}}(\varphi)-[\tilde{e},\overline{\tilde{e}}]^{(0,1)}(\varphi)\\
&  =  1+\frac{1}{2}(V_{1}V_{1}(\varphi)+(JV_{1})(JV_{1})(\varphi)+\sqrt{-1}[V_{1},JV_{1}](\varphi))-[\tilde{e},\overline{\tilde{e}}]^{(0,1)}(\varphi)\\
&  =  \frac{1}{2}\left(\lambda_{1}+\sum_{i\alpha>1}\lambda_{\alpha}\mu_{\alpha}^{2}\right)+1+(\nabla_{V_{1}}V_{1})(\varphi)+(\nabla_{JV_{1}}JV_{1})(\varphi)\\
&  \quad\,   +\sqrt{-1}[V_{1},JV_{1}](\varphi) -[\tilde{e},\overline{\tilde{e}}]^{(0,1)}(\varphi)\\
& \leq \frac{1}{2}\left(\lambda_{1}+\sum_{\alpha>1}\lambda_{\alpha}\mu_{\alpha}^{2}\right)+C.
\end{split}
\end{equation*}
Note  $\lambda_{1}\geq\frac{C_{A}}{\ve}$.   Thus we deduce
\begin{equation*}
\left(1-\frac{C_{A}}{\lambda_{1}^{2}}\right)\eta_{1}\leq \frac{1}{2}\left(\lambda_{1}+\sum_{\alpha>1}\lambda_{\alpha}\mu_{\alpha}^{2}\right)+C.
\end{equation*}
As a consequence,
\begin{equation*}
\eta_{1} \leq \frac{1}{2}\left(\lambda_{1}+\sum_{\alpha>1}\lambda_{\alpha}\mu_{\alpha}^{2}\right)+C+\frac{C_{A}}{\lambda_{1}}\cdot\frac{\eta_{1}}{\lambda_{1}}
         \leq \frac{1}{2}\left(\lambda_{1}+\sum_{\alpha>1}\lambda_{\alpha}\mu_{\alpha}^{2}\right)+C.
\end{equation*}
Hence, for $i\geq 2$, we obtain
\begin{equation*}
(1-\ve)G^{i\overline{i}} = (1-\ve)\frac{\sigma_{1}(i)}{\sigma_{2}}\leq\frac{\eta_{1}}{\sigma_{2}}-\frac{\ve\eta_{1}}{\sigma_{2}}+C
\leq \frac{1}{2\sigma_{2}}\left(\lambda_{1}+\sum_{\alpha>1}\lambda_{\alpha}\mu_{\alpha}^{2}\right),
\end{equation*}
where we used (\ref{Lemma 2 equation 2}) and $\lambda_{1}\geq\frac{C_{A}}{\ve}$ in the last inequality.
\end{proof}

At $x_{0}$, we assume that the eigenvalues of matrix $(-G^{i\overline{i},j\overline{j}})$ are
\begin{equation*}
\kappa_{1}\geq\kappa_{2}\cdots\geq\kappa_{n}.
\end{equation*}
Let $\xi_{i}=(\xi_{i}^{1},\xi_{i}^{2},\cdots,\xi_{i}^{n})$ be the $g$-unit eigenvector corresponding to $\kappa_{i}$ for $i=1,2,\cdots,n$.
Some estimates for eigenvalues $\kappa_{i}$ and its eigenvectors $\xi_{i}$ are given in the following lemma, which plays important role in the estimate of $II_{3}$.

\begin{lemma}\label{Lemma 6 and Lemma 7}
\begin{enumerate}[(1)]
  \item $C_{A}^{-1}\lambda_{1}^{-2}\leq \kappa_{n}\leq C_{A}\lambda_{1}^{-2}$ and $\kappa_{i}\geq C_{A}^{-1}$ for $i\leq n-1$.\\
  \item $\sum_{i=2}^{n}|\xi_{n}^{i}|^{2}\leq C_{A}\lambda_{1}^{-2}$.
\end{enumerate}
\end{lemma}

\begin{proof}
Since the proof of Lemma \ref{Lemma 6 and Lemma 7} is a little tedious, we give it in Appendix.
\end{proof}

Now we begin to estimate $II_{3}$.

\begin{lemma}\label{Lemma 4}
For any positive number $\gamma>0$, we have
\begin{equation*}
\begin{split}
&II_{3}\\
 & \leq \frac{C_{A}}{\ve}\sum_{i\geq2}\sum_{q\geq 2}G^{i\overline{i}}\frac{|V_{1}(\tilde{g}_{i\overline{q}})|^{2}}{\lambda_{1}^{4}}
         +\frac{C}{\ve}\sum_{i}G^{i\overline{i}}+2(1-\ve)(1+\gamma)\sum_{i\geq2}G^{i\overline{i}}
          \frac{|V_{1}(\tilde{g}_{{i}\overline{1}})|^{2}}{\lambda_{1}^{2}}\\
       & \quad  +(1-\ve)\left(1+\frac{1}{\gamma}\right)\left(\lambda_{1}-\sum_{\alpha>1}\lambda_{\alpha}\mu_{\alpha}^{2}\right)
          \left(\sum_{i\geq2}\sum_{\alpha>1}\frac{G^{i\overline{i}}}{\lambda_{1}^{2}}\frac{|e_{i}(\varphi_{V_{\alpha}V_{1}})|^2}{\lambda_{1}-\lambda_{\alpha}}\right).
\end{split}
\end{equation*}
\end{lemma}

\begin{proof}
By the relation (\ref{Lemma 4 equation 1}) and the Cauchy-Schwarz inequality, we have
\begin{equation}\label{Lemma 4 equation 2}
\begin{split}
&II_{3} \\
    & = (1-2\ve)\sum_{i\geq2}\frac{G^{i\overline{i}}|\sqrt{2}\sum_{q} \overline{\nu_{q}}V_{1}(\tilde{g}_{i\overline{q}})-
          \sqrt{-1}\sum_{\alpha>1}\mu_{\alpha}e_{i}(\varphi_{V_{1}V_{\alpha}})+E|^{2}}{\lambda_{1}^{2}}\\
    & = (1-\ve)\sum_{i\geq2}\frac{G^{i\overline{i}}|\sqrt{2} \overline{\nu_{1}}V_{1}(\tilde{g}_{i\overline{1}})-
          \sqrt{-1}\sum_{\alpha>1}\mu_{\alpha}e_{i}(\varphi_{V_{1}V_{\alpha}})|^{2}}{\lambda_{1}^{2}}\\
    & \quad +\frac{C}{\ve}\sum_{i\geq2}\frac{G^{i\overline{i}}|\sqrt{2}\sum_{q\geq2} \overline{\nu_{q}}V_{1}(\tilde{g}_{i\overline{q}})+E|^{2}}{\lambda_{1}^{2}}  \\
    & \leq(1-\ve)\sum_{i\geq2}\frac{G^{i\overline{i}}|\sqrt{2}\overline{\nu_{1}}V_{1}(\tilde{g}_{i\overline{1}})-\sqrt{-1}\sum_{\alpha>1}\mu_{\alpha}
          e_{i}(\varphi_{V_{1}V_{\alpha}})|^{2}}{\lambda_{1}^{2}}\\
    & \quad +\frac{C_{A}}{\ve}\sum_{i\geq2}
         \sum_{q\geq 2}\frac{G^{i\overline{i}}|V_{1}(\tilde{g}_{i\overline{q}})|^{2}}{\lambda_{1}^{4}}+\frac{C}{\ve}\sum_{i}G^{i\overline{i}}.
\end{split}
\end{equation}
Here we used Lemma \ref{Lemma 3} in  the last inequality.  On the other hand,  by the Cauchy-Schwarz inequality, we have
\begin{equation*}
\begin{split}
&(1-  \ve)\sum_{i\geq 2}\frac{G^{i\overline{i}}|\sqrt{2}\overline{\nu_{1}}V_{1}(\tilde{g}_{{i}\overline{1}})-\sum_{\alpha>1}\mu_{\alpha}
                 e_{i}(\varphi_{V_{1}V_{\alpha}})|^{2}}{\lambda_{1}^{2}}\\
              &\leq  2(1-\ve)(1+\gamma)\sum_{i\geq2}\frac{G^{i\overline{i}}|V_{1}(\tilde{g}_{{i}\overline{1}})|^{2}}{\lambda_{1}^{2}}\\
             & +(1-\ve)\left(1+\frac{1}{\gamma}\right)\sum_{i\geq2}\frac{G^{i\overline{i}}|\sum_{\alpha>1}\mu_{\alpha}e_{i}
                 (\varphi_{V_{1}V_{\alpha}})|^{2}}{\lambda_{1}^{2}},
\end{split}
\end{equation*}
and
\begin{equation*}
\begin{split}
\left|\sum_{\alpha>1}\mu_{\alpha} e_{i}(\varphi_{V_{\alpha}V_{1}})\right|^{2}
& \leq\left(\sum_{\alpha>1}(\lambda_{1}-\lambda_{\alpha})\mu_{\alpha}^{2}\right)
  \left(\sum_{\alpha>1}\frac{|e_{i}(\varphi_{V_{\alpha}V_{1}})|^2}{\lambda_{1}-\lambda_{\alpha}}\right)\\
& =\left(\lambda_{1}-\sum_{\alpha>1}\lambda_{\alpha}\mu_{\alpha}^{2}\right)
  \left(\sum_{\alpha>1}\frac{|e_{i}(\varphi_{V_{\alpha}V_{1}})|^2}{\lambda_{1}-\lambda_{\alpha}}\right).
\end{split}
\end{equation*}
 Thus
\begin{equation*}
\begin{split}
& (1-\ve)\sum_{i\geq2}\frac{G^{i\overline{i}}|\sqrt{2}\overline{\nu_{1}}V_{1}(\tilde{g}_{i\overline{1}})-\sqrt{-1}\sum_{\alpha>1}\mu_{\alpha}
e_{i}(\varphi_{V_{1}V_{\alpha}})|^{2}}{\lambda_{1}^{2}}\\
  \leq ~~&  2(1-\ve)(1+\gamma)\sum_{i\geq2}\frac{G^{i\overline{i}}|V_{1}(\tilde{g}_{{i}\overline{1}})|^{2}}{\lambda_{1}^{2}}\\
    &  +(1-\ve)\left(1+\frac{1}{\gamma}\right)\left(\lambda_{1}-\sum_{\alpha>1}\lambda_{\alpha}\mu_{\alpha}^{2}\right)
       \left(\sum_{i\geq2}\sum_{\alpha>1}\frac{G^{i\overline{i}}}{\lambda_{1}^{2}}\frac{|e_{i}(\varphi_{V_{\alpha}V_{1}})|^2}{\lambda_{1}-\lambda_{\alpha}}\right).
\end{split}
\end{equation*}
Inserting the above inequality  into (\ref{Lemma 4 equation 2}), the lemma is proved.
\end{proof}

\begin{lemma}\label{Lemma 8}
At $x_{0}$, if $\lambda_{1}\geq\frac{C_{A}}{\ve}$, then we have
\begin{equation*}
-\frac{G^{i\overline{i},k\overline{k}}V_{1}(\tilde{g}_{k\overline{k}})V_{1}(\tilde{g}_{i\overline{i}})}{\lambda_{1}}
\geq \frac{C_{A}}{\ve}\sum_{i\geq2}G^{i\overline{i}}\frac{|V_{1}(\tilde{g}_{i\overline{i}})|^{2}}{\lambda_{1}^{4}}.
\end{equation*}
\end{lemma}

\begin{proof}
Recall that $\xi_{i}=(\xi_{i}^{1},\xi_{i}^{2},\cdots,\xi_{i}^{n})$ are the $g$-unit eigenvector corresponding to $\kappa_{i}$ for $i=1,2,\cdots,n$.
Then there are  complex numbers $\tau_{1},\tau_{2},\cdots,\tau_{n}$ such that
\begin{equation*}
V_{1}(\tilde{g}_{i\overline{i}})=\sum_{q=1}^{n}\tau_{q}\xi_{q}^{i} ~ \text{~for $i=1,2,\cdots,n$.}
\end{equation*}
Since
\begin{equation*}
G^{i\overline{i}}=\frac{\sigma_{1}(i)}{\sigma_{2}}\leq C\lambda_{1} \text{~for $i=1,2,\cdots,n$},
\end{equation*}
  we derive
\begin{equation}\label{Lemma 8 equation 4}
-\frac{C_{A}}{\ve}\sum_{i\geq2}G^{i\overline{i}}\frac{|V_{1}(\tilde{g}_{i\overline{i}})|^{2}}{\lambda_{1}^{4}}
\geq -\sum_{i\geq2}\sum_{q=1}^{n}\frac{C_{A}}{\ve\lambda_{1}^{3}}|\tau_{q}|^{2}|\xi_{q}^{i}|^{2}.
\end{equation}
Also we have
\begin{equation}\label{Lemma 8 equation 5}
-\frac{G^{i\overline{i},k\overline{k}}V_{1}(\tilde{g}_{k\overline{k}})V_{1}(\tilde{g}_{i\overline{i}})}{\lambda_{1}}
=\frac{1}{\lambda_{1}}\sum_{q=1}^{n}\kappa_{q}|\tau_{q}|^{2}.
\end{equation}
By (\ref{Lemma 8 equation 4}) and (\ref{Lemma 8 equation 5}), we obtain
\begin{equation}\label{Lemma 8 equation 6}
\begin{split}
&-\frac{C_{A}}{\ve}\sum_{i\geq2}  G^{i\overline{i}}\frac{|V_{1}(\tilde{g}_{i\overline{i}})|^{2}}{\lambda_{1}^{4}}
  -\frac{G^{i\overline{i},k\overline{k}}V_{1}(\varphi_{k\overline{k}})V_{1}(\varphi_{i\overline{i}})}{\lambda_{1}}\\
& \ge -\sum_{i\geq2}\sum_{q=1}^{n}\frac{C_{A}}{\ve\lambda_{1}^{3}}|\tau_{q}|^{2}|\xi_{q}^{i}|^{2}+\frac{1}{\lambda_{1}}
\sum_{q=1}^{n}\kappa_{q}|\tau_{q}|^{2}\\
&  = -\sum_{i\geq2}\frac{C_{A}}{\ve\lambda_{1}^{3}}|\tau_{n}|^{2}|\xi_{n}^{i}|^{2}+\frac{1}{\lambda_{1}}\kappa_{n}|\tau_{n}|^{2} \\
& +\sum_{q=1}^{n-1}\left( -\sum_{i\geq2}\frac{C_{A}}{\ve\lambda_{1}^{3}}|\tau_{q}|^{2}|\xi_{q}^{i}|^{2}+\frac{1}{\lambda_{1}}\kappa_{q}|\tau_{q}|^{2} \right).
\end{split}
\end{equation}
Moreover, by $\lambda_{1}\geq\frac{C_{A}}{\ve}$ and Lemma \ref{Lemma 6 and Lemma 7}, we see
\begin{equation*}
\begin{split}
-\sum_{i\geq2}\frac{C_{A}}{\ve\lambda_{1}^{3}}|\tau_{n}|^{2}|\xi_{n}^{i}|^{2}+\frac{1}{\lambda_{1}}\kappa_{n}|\tau_{n}|^{2}
& \geq\left(\frac{1}{C_{A}\lambda_{1}^{3}}-\frac{C_{A}}{\ve\lambda_{1}^{5}}\right)|\tau_{n}|^{2}
  \geq 0,\\
\sum_{q=1}^{n-1}\left( -\sum_{i\geq2}\frac{C_{A}}{\ve\lambda_{1}^{3}}|\tau_{q}|^{2}|\xi_{q}^{i}|^{2}+\frac{1}{\lambda_{1}}\kappa_{q}|\tau_{q}|^{2} \right)
& \geq\sum_{q=1}^{n-1}\left(\frac{1}{C_{A}\lambda_{1}}-\frac{C_{A}}{\ve\lambda_{1}^{3}}\right)|\tau_{q}|^{2}
  \geq 0.
\end{split}
\end{equation*}
 Thus Lemma \ref{Lemma 8} follows from the above inequalities and (\ref{Lemma 8 equation 6}).
\end{proof}

Lemma  \ref{Lemma 8} gives an estimate for the term   $\frac{C_{A}}{\ve}\sum_{i\geq2}G^{i\overline{i}}\frac{|V_{1}(\tilde{g}_{i\overline{i}})|^{2}}{\lambda_{1}^{4}} $ in Lemma \ref{Lemma 4}. We need to deal with other terms there.
By the definition of $\lambda_{\alpha}$ and $\mu_{\alpha}$, it is clear that $\lambda_{1}-\sum_{\alpha>1}\lambda_{\alpha}\mu_{\alpha}^{2}>0$. From Lemma \ref{Lemma 9}, we  see $\lambda_{1}+\sum_{\alpha>1}\lambda_{\alpha}\mu_{\alpha}^{2}>0$. Recalling that the constant $\gamma>0$ in Lemma \ref{Lemma 4} is arbitrary, now we choose
\begin{equation*}
\gamma=\frac{\lambda_{1}-\sum_{\alpha>1}\lambda_{\alpha}\mu_{\alpha}^{2}}{\lambda_{1}+\sum_{\alpha>1}\lambda_{\alpha}\mu_{\alpha}^{2}}.
\end{equation*}
Thus by  Lemma \ref{Lemma 9} and the definition of $\gamma$, we obtain
\begin{equation}\label{Lemma 10 equation 2}
2\sum_{i\geq 2}\frac{|V_{1}(\tilde{g}_{i\overline{1}})|^{2}}{\sigma_{2}\lambda_{1}}
\geq 2(1-\ve)(1+\gamma)\sum_{i\geq2}G^{i\overline{i}}\frac{|V_{1}(\tilde{g}_{{i}\overline{1}})|^{2}}{\lambda_{1}^{2}}
\end{equation}
and
\begin{equation}\label{Lemma 10 equation 3}
\begin{split}
&   (2-2\ve)\sum_{i\geq2}\sum_{\alpha>1}\frac{G^{i\overline{i}}|e_{i}(\varphi_{V_{\alpha}V_{1}})|^{2}}{\lambda_{1}(\lambda_{1}-\lambda_{\alpha})}\\
= ~~& (1-\ve)\left(1+\frac{1}{\gamma}\right)\left(\lambda_{1}-\sum_{\alpha>1}\lambda_{\alpha}\mu_{\alpha}^{2}\right)
\sum_{i\geq2}\sum_{\alpha>1}\frac{G^{i\overline{i}}}{\lambda_{1}^{2}}\frac{|e_{i}(\varphi_{V_{\alpha}V_{1}})|^2}{\lambda_{1}-\lambda_{\alpha}}.
\end{split}
\end{equation}

\begin{lemma}\label{Lemma 10}
At $x_{0}$, if $\lambda_{1}\geq\frac{C_{A}}{\ve}$, we have
\begin{equation*}
II_{3}\leq I+\frac{C}{\ve}\sum_{i}G^{i\overline{i}}.
\end{equation*}
\end{lemma}

\begin{proof}
By the definition of $G^{i\overline{k},k\overline{i}}$ (see (\ref{Definition of Fijkl})), it is clear that
\begin{equation*}
-\sum_{k\neq l}\frac{G^{k\overline{l},l\overline{k}}|V_{1}(\tilde{g}_{k\overline{l}})|^{2}}{\lambda_{1}}
= \sum_{k\neq l}\frac{|V_{1}(\tilde{g}_{k\overline{l}})|^{2}}{\sigma_{2}\lambda_{1}}
= 2\sum_{i\geq 2}\frac{|V_{1}(\tilde{g}_{i\overline{1}})|^{2}}{\sigma_{2}\lambda_{1}}
    +\sum_{i\geq2}\sum_{q\geq2,q\neq i}\frac{|V_{1}(\tilde{g}_{i\overline{q}})|^{2}}{\sigma_{2}\lambda_{1}}.
\end{equation*}
On the other hand, by (\ref{Lemma 2 equation 2}),  we see
\begin{equation*}\label{Lambda1 and eta1}
C_{A}^{-1}\eta_{1}\leq \lambda_{1}\leq C_{A}\eta_{1}.
\end{equation*}
Note  $\lambda_{1}\geq\frac{C_{A}}{\ve}$.  Then  by Lemma \ref{Lemma 2}, we have
\begin{equation*}
G^{i\overline{i}}=\frac{\sigma_{1}(i)}{\sigma_{2}}\leq \frac{\eta_{1}+C}{\sigma_{2}} \leq \frac{\ve\lambda_{1}^{3}}{C_{A}\sigma_{2}} \text{~for $i\geq 2$.}
\end{equation*}
This implies
\begin{equation*}
\sum_{i\geq2}\sum_{q\geq2,q\neq i}\frac{|V_{1}(\tilde{g}_{i\overline{q}})|^{2}}{\sigma_{2}\lambda_{1}}
\geq \frac{C_{A}}{\ve}\sum_{i\geq2}\sum_{q\geq2,q\neq i}G^{i\overline{i}}\frac{|V_{1}(\tilde{g}_{i\overline{q}})|^{2}}{\lambda_{1}^{4}}.
\end{equation*}
Hence by (\ref{Lemma 10 equation 2}), we deduce
\begin{align}\label{lemma4.4-1}
&-\sum_{k\neq l}\frac{G^{k\overline{l},l\overline{k}}|V_{1}(\tilde{g}_{k\overline{l}})|^{2}}{\lambda_{1}}\notag\\
&\ge 2(1-\ve)(1+\gamma)\sum_{i\geq2}G^{i\overline{i}}\frac{|V_{1}(\tilde{g}_{{i}\overline{1}})|^{2}}{\lambda_{1}^{2}}
+\frac{C_{A}}{\ve}\sum_{i\geq2}\sum_{q\geq2,q\neq i}G^{i\overline{i}}\frac{|V_{1}(\tilde{g}_{i\overline{q}})|^{2}}{\lambda_{1}^{4}}.
\end{align}
By (\ref{Lemma 10 equation 3}) and (\ref{lemma4.4-1}), we see
\begin{equation}\label{Lemma 10 equation 5}
\begin{split}
& (2-\ve)\sum_{\alpha>1}\frac{G^{i\overline{i}}|e_{i}(\varphi_{V_{\alpha}V_{1}})|^{2}}{\lambda_{1}(\lambda_{1}-\lambda_{\alpha})}
      -\sum_{k\neq l}\frac{G^{k\overline{l},l\overline{k}}|V_{1}(\tilde{g}_{k\overline{l}})|^{2}}{\lambda_{1}}\\
\geq~~& 2(1-\ve)(1+\gamma)\sum_{i\geq2}G^{i\overline{i}}\frac{|V_{1}(\tilde{g}_{{i}\overline{1}})|^{2}}{\lambda_{1}^{2}}
+\frac{C_{A}}{\ve}\sum_{i\geq2}\sum_{q\geq2,q\neq i}G^{i\overline{i}}\frac{|V_{1}(\tilde{g}_{i\overline{q}})|^{2}}{\lambda_{1}^{4}}\\
&  +(1-\ve)\left(1+\frac{1}{\gamma}\right)\left(\lambda_{1}-\sum_{\alpha>1}\lambda_{\alpha}\mu_{\alpha}^{2}\right)
\sum_{i\geq2}\sum_{\alpha>1}\frac{G^{i\overline{i}}}{\lambda_{1}^{2}}\frac{|e_{i}(\varphi_{V_{\alpha}V_{1}})|^2}{\lambda_{1}-\lambda_{\alpha}}.
\end{split}
\end{equation}
Then Lemma \ref{Lemma 10} follows from Lemma \ref{Lemma 4}, Lemma \ref{Lemma 8} and (\ref{Lemma 10 equation 5}).
\end{proof}

Combining Lemma \ref{II1 and II2} and Lemma \ref{Lemma 10}, we finally obtain

\begin{proposition}\label{Lemma 11}
If $\lambda_{1}\geq \frac{C_{A}}{\ve}$, we have
\begin{equation*}
\begin{split}
II & = II_{1}+II_{2}+II_{3}\\
& \leq I+12\ve A^{2}e^{-2A\vp}G^{i\overline{i}}|e_{i}(\varphi)|^{2}+2(h')^{2}G^{i\overline{i}}|e_{i}(|{\partial\varphi}|_{g}^{2})|^{2}
       +\frac{C}{\ve}\sum_{i}G^{i\ol{i}}+C_{A}.
\end{split}
\end{equation*}
\end{proposition}

By Proposition \ref{LQ Lemma} and Proposition \ref{Lemma 11}, we can complete the proof of  Theorem \ref{Generalized second order estimate}.

\begin{proof}[Proof of Theorem \ref{Generalized second order estimate}]
Without loss of generality, we assume that $\sup_{M}\vp=0$.
Then by  Proposition \ref{LQ Lemma} and Proposition \ref{Lemma 11},  we see that at $x_{0}$ there exists a uniform constant $C_{1}$ such that
\begin{equation}\label{main-inequality-2}
\begin{split}
0 & \geq  \left(\ve_{0}Ae^{-A\varphi}-\frac{C_{1}}{\ve}\right)\sum_{i}G^{i\overline{i}}+\frac{h'}{4}\sum_{k}G^{i\overline{i}}(|e_{i}e_{k}(\varphi)|^{2}+
          |e_{i}\overline{e}_{k}(\varphi)|^{2})\\
  &\quad  +(A^{2}e^{-A\varphi}-12\ve A^{2}e^{-2A\varphi})G^{i\overline{i}}|e_{i}(\varphi)|^{2}-C_{1}Ae^{-A\varphi}.
\end{split}
\end{equation}
Choose $A=12C_{1}+1$ and $\ve=\frac{e^{A\vp(x_{0})}}{12}\in(0,\frac{1}{12}]$ so that
\begin{equation*}
Ae^{-A\vp}-\frac{C_{1}}{\ve}\geq 1 \text{~and~}
A^{2}e^{-A\vp}-12\ve A^{2}e^{-2A(\vp)}\geq0.
\end{equation*}
We get from (\ref{main-inequality-2}),
\begin{equation*}
\sum_{i}G^{i\overline{i}}+\frac{h'}{4}\sum_{k}G^{i\overline{i}}(|e_{i}e_{k}(\varphi)|^{2}+|e_{i}\overline{e}_{k}(\varphi)|^{2})\leq C.
\end{equation*}
As a consequence,  $\sum_{i}G^{i\overline{i}}\leq C$. Combining this with Maclaurin's inequality, we obtain (for more details, cf. \cite[Lemma 2.2]{HMW10}),
\begin{equation*}
G^{i\overline{i}}\geq C^{-1} \text{~for $i=1,2,\cdots,n$.}
\end{equation*}
Thus we get
\begin{equation*}
\lambda_{1}^{2}\leq C \sum_{k}G^{i\overline{i}}(|e_{i}e_{k}(\varphi)|^{2}+|e_{i}\overline{e}_{k}(\varphi)|^{2})\leq C,
\end{equation*}
as required.
\end{proof}

\section{Appendix}
In this appendix,  we give a  proof of Lemma \ref{Lemma 6 and Lemma 7}. Here we use the same notations in Section 4. We need the following algebraic  Lemma for $\sigma_2$ polynomial function.

\begin{lemma}\label{Lemma 5}
At $x_{0}$, we have
\begin{equation*}
\det(-G^{i\overline{i},j\overline{j}})=(n-1)\sigma_{2}^{-n}
\end{equation*}
\end{lemma}

\begin{proof}
For convenience, we define $\vec{\sigma}=(\sigma_{1}(1),\cdots,\sigma_{1}(n))$ and $M_{1}=\vec{\sigma}^{T}\vec{\sigma}$, where $\vec{\sigma}^{T}$ denotes the transpose of the vector $\vec{\sigma}$. By the definition of $G^{i\overline{i},j\overline{j}}$, it is clear that
\begin{equation}\label{Lemma 5 equation 1}
(-\sigma_{2}^{2}G^{i\overline{i},j\overline{j}})=M_{1}-M_{2},
\end{equation}
where
\begin{equation*}
M_{2}=\left(
\begin{matrix}
   0          & \sigma_{2}    & \sigma_{2}  & \cdots  & \sigma_{2}   \\
   \sigma_{2} &  0            & \sigma_{2}  & \cdots  & \sigma_{2}   \\
   \vdots     & \vdots        & \vdots      &  ~      & \vdots       \\
   \sigma_{2} &  \sigma_{2}   & \sigma_{2}  & \cdots  & 0            \\
\end{matrix}
\right).
\end{equation*}
Since the rank of matrix $M_{1}$ is one, any two columns of $M_{1}$ are proportional. Combining this and properties of the determinant, we have
\begin{equation}\label{Lemma 5 equation 2}
\det (M_{1}-M_{2})=\sum_{i=1}^{n}\det A_{i}+(-1)^{n}\det M_{2},
\end{equation}
where
\begin{equation*}
\begin{split}
& \qquad\qquad\qquad\qquad\qquad\qquad\ \text{the $i$-th column}\\
& A_{i}=\left(
\begin{matrix}
   0           & -\sigma_{2}   & -\sigma_{2}   & \cdots   & \sigma_{1}(1)\sigma_{1}(i)  & \cdots  & -\sigma_{2}   \\
   -\sigma_{2} &  0            & -\sigma_{2}   & \cdots   & \sigma_{1}(2)\sigma_{1}(i)  & \cdots  & -\sigma_{2}   \\
   \vdots      & \vdots        & \vdots        &  ~       & \vdots                      &   ~     & \vdots        \\
   -\sigma_{2} & -\sigma_{2}   & -\sigma_{2}   & \cdots   & \sigma_{1}(n)\sigma_{1}(i)  & \cdots  & 0             \\
\end{matrix}
\right).
\end{split}
\end{equation*}
Applying some elementary row operations to $A_{i}$, we obtain
\begin{equation*}
\det A_{i}=\sigma_{1}(i)\sigma_{2}^{n-1}\left(\sum_{k=1}^{n}\sigma_{1}(k)-(n-1)\sigma_{1}(i)\right).
\end{equation*}
Therefore,
\begin{equation}\label{Lemma 5 equation 3}
\begin{split}
\sum_{i=1}^{n}\det{(A_{i})} & = \sum_{i=1}^{n}\sigma_{1}(i)\sigma_{2}^{n-1}\left(\sum_{k=1}^{n}\sigma_{1}(k)-(n-1)\sigma_{1}(i)\right)\\
                      & = \sigma_{2}^{n-1}\left(\left(\sum_{i=1}^{n}\sigma_{1}(i)\right)^{2}-\sum_{i=1}^{n}(n-1)(\sigma_{1}(i))^{2}\right)\\
                      & = \sigma_{2}^{n-1}\left((n-1)^2\sigma_{1}^{2}-(n-1)\sum_{i}(\sigma_{1}-\eta_{i})^{2}\right)\\
                      & = (n-1)\sigma_{2}^{n-1}\left(\sigma_{1}^{2}-\sum_{i}\eta_{i}^{2}\right)\\
                      & = 2(n-1)\sigma_{2}^{n}.
\end{split}
\end{equation}
On the other hand, it is clear that
\begin{equation}\label{Lemma 5 equation 4}
\det M_{2}=(-1)^{n-1}(n-1)\sigma_{2}^{n}.
\end{equation}
Then Lemma \ref{Lemma 5} follows from (\ref{Lemma 5 equation 1}), (\ref{Lemma 5 equation 2}), (\ref{Lemma 5 equation 3}) and (\ref{Lemma 5 equation 4}).
\end{proof}

\begin{proof}[Proof of (1) in Lemma \ref{Lemma 6 and Lemma 7}]
Let $a_{1}\geq a_{2}\geq\cdots\geq a_{n}$ and $b_{1}\geq b_{2}\geq\cdots\geq b_{n}$ be the eigenvalues of $M_{1}$ and $M_{2}$, respectively.  Then
\begin{equation*}
a_{1}=\|\vec{\sigma}\|^{2}, ~a_{2}=a_{3}=\cdots=a_{n}=0
\end{equation*}
and
\begin{equation*}
b_{1}=(n-1)\sigma_{2}, ~b_{2}=b_{3}=\cdots=b_{n}=-\sigma_{2}.
\end{equation*}
By Weyl's inequality in matrix theory (cf. \cite[Theorem 4.3.1]{HJ02}), we see
\begin{equation*}
\frac{a_{1}}{\sigma_{2}^{2}}-\frac{b_{1}}{\sigma_{2}^{2}}\leq\kappa_{1}\leq\frac{a_{1}}{\sigma_{2}^{2}}-\frac{b_{n}}{\sigma_{2}^{2}}
\end{equation*}
and
\begin{equation*}
\kappa_{i}\leq\frac{a_{i}}{\sigma_{2}^{2}}-\frac{b_{n}}{\sigma_{2}^{2}} \text{~for $i\geq2$}.
\end{equation*}
It follows
\begin{equation}\label{Lemma 6 equation 1}
\text{$C_{A}^{-1}\lambda_{1}^{2}\leq\kappa_{1}\leq C_{A}\lambda_{1}^{2}$ and $\kappa_{i}\leq C_{A}$ for $i\geq 2$.}
\end{equation}
Thus by  Lemma \ref{Lemma 5},   we  get
\begin{equation}\label{Lemma 6 equation 2}
\kappa_{n}=\frac{\det(-G^{i\overline{i},j\overline{j}})}{\kappa_{1}\kappa_{2}\cdots\kappa_{n-1}}\geq\frac{1}{C_{A}\lambda_{1}^{2}}.
\end{equation}

On the other hand, since $\kappa_{n}$ is the smallest eigenvalue of matrix $(-G^{i\overline{i},j\overline{j}})$, by (\ref{Definition of Fijkl}) and  Corollary  \ref{Lemma 2},  we have
\begin{equation}\label{Lemma 6 equation 3}
\kappa_{n}\leq -G^{1\overline{1},1\overline{1}}=\frac{(\sigma_{1}(1))^{2}}{\sigma_{2}^{2}}
=\left(\frac{\sigma_{2}-\sum_{i>j>1}\eta_{i}\eta_{j}}{\eta_{1}\sigma_{2}}\right)^{2}\leq\frac{1}{C_{A}\lambda_{1}^{2}}.
\end{equation}
Then by  (\ref{Lemma 6 equation 3}) and  (\ref{Lemma 6 equation 1}),  we  have
\begin{equation}\label{Lemma 6 equation 4}
\kappa_{i}\geq\frac{\det(-G^{i\overline{i},j\overline{j}})}{\kappa_{1}\kappa_{2}^{n-3}\kappa_{n}}\geq C_{A}^{-1},~ \forall~ i\leq n-1.
\end{equation}
The first part  (1)  of  Lemma \ref{Lemma 6 and Lemma 7} is proved.
\end{proof}

\begin{proof}[Proof of (2) in Lemma \ref{Lemma 6 and Lemma 7}]
For simplicity, we prove the case when $n=4$. The general case can be proved by  the same way.

Recall that the vector $\xi_{4}$ is the eigenvector of matrix $(-G^{i\overline{i},j\overline{j}})$ corresponding to $\kappa_{4}$.  We   use the following  elementary row operation   of  $(-G^{i\overline{i},j\overline{j}})$  to compute the  components  $\xi_{4}^{i}$ of $\xi_{4}$,
\begin{equation*}
\begin{split}
(\kappa_{4}I_{4}+ & G^{i\overline{i},j\overline{j}})=\\
&\left(
\begin{matrix}
\kappa_{4}-\frac{(\sigma_{1}(1))^{2}}{\sigma_{2}^{2}}           &     \frac{\sigma_{2}-\sigma_{1}(1)\sigma_{1}(2)}{\sigma_{2}^{2}}   & \frac{\sigma_{2}-\sigma_{1}(1)\sigma_{1}(3)}{\sigma_{2}^{2}}    &     \frac{\sigma_{2}-\sigma_{1}(1)\sigma_{1}(4)}{\sigma_{2}^{2}}   \\[3mm]
\frac{\sigma_{2}-\sigma_{1}(2)\sigma_{1}(1)}{\sigma_{2}^{2}}    &     \kappa_{4}-\frac{(\sigma_{1}(2))^{2}}{\sigma_{2}^{2}}          &
\frac{\sigma_{2}-\sigma_{1}(2)\sigma_{1}(3)}{\sigma_{2}^{2}}    &     \frac{\sigma_{2}-\sigma_{1}(2)\sigma_{1}(4)}{\sigma_{2}^{2}}   \\[3mm]
\frac{\sigma_{2}-\sigma_{1}(3)\sigma_{1}(1)}{\sigma_{2}^{2}}    &     \frac{\sigma_{2}-\sigma_{1}(3)\sigma_{1}(2)}{\sigma_{2}^{2}}   &
\kappa_{4}-\frac{(\sigma_{1}(3))^{2}}{\sigma_{2}^{2}}           &     \frac{\sigma_{2}-\sigma_{1}(3)\sigma_{1}(4)}{\sigma_{2}^{2}}   \\[3mm]
\frac{\sigma_{2}-\sigma_{1}(4)\sigma_{1}(1)}{\sigma_{2}^{2}}    &     \frac{\sigma_{2}-\sigma_{1}(4)\sigma_{1}(2)}{\sigma_{2}^{2}}   &
\frac{\sigma_{2}-\sigma_{1}(3)\sigma_{1}(4)}{\sigma_{2}^{2}}    &     \kappa_{4}-\frac{(\sigma_{1}(4))^{2}}{\sigma_{2}^{2}}
\end{matrix}
\right),
\end{split}
\end{equation*}
where $I_{4}$ denotes the identity matrix. There are four steps.

\bigskip

\noindent
{\bf Step 1.} For $i=1,2,3$, multiplying the $4$-th row by $-\frac{\sigma_{1}(i)}{\sigma_{1}(4)}$, and adding that to the $i$-th row, we obtain
\begin{equation*}
\left(
\begin{matrix}
\kappa_{4}-\frac{(\sigma_{1}(1))}{\sigma_{1}(4)\sigma_{2}}     &     \frac{\sigma_{1}(4)-\sigma_{1}(1)}{\sigma_{1}(4)\sigma_{2}}                      & \frac{\sigma_{1}(4)-\sigma_{1}(1)}{\sigma_{1}(4)\sigma_{2}}    &     \frac{\sigma_{1}(4)-\sigma_{1}(1)\sigma_{2}\kappa_{4}}{\sigma_{1}(4)\sigma_{2}}  \\[3mm]
\frac{\sigma_{1}(4)-\sigma_{1}(2)}{\sigma_{1}(4)\sigma_{2}}    &     \kappa_{4}-\frac{(\sigma_{1}(2))}{\sigma_{1}(4)\sigma_{2}}                       &
\frac{\sigma_{1}(4)-\sigma_{1}(2)}{\sigma_{1}(4)\sigma_{2}}    &     \frac{\sigma_{1}(4)-\sigma_{1}(2)\sigma_{2}\kappa_{4}}{\sigma_{1}(4)\sigma_{2}}  \\[3mm]
\frac{\sigma_{1}(4)-\sigma_{1}(3)}{\sigma_{1}(4)\sigma_{2}}    &     \frac{\sigma_{1}(4)-\sigma_{1}(3)}{\sigma_{1}(4)\sigma_{2}}                      &
\kappa_{4}-\frac{(\sigma_{1}(3))}{\sigma_{1}(4)\sigma_{2}}     &     \frac{\sigma_{1}(4)-\sigma_{1}(3)\sigma_{2}\kappa_{4}}{\sigma_{1}(4)\sigma_{2}}  \\[3mm]
\frac{\sigma_{2}-\sigma_{1}(4)\sigma_{1}(1)}{\sigma_{2}^{2}}   &     \frac{\sigma_{2}-\sigma_{1}(4)\sigma_{1}(2)}{\sigma_{2}^{2}}                     &
\frac{\sigma_{2}-\sigma_{1}(3)\sigma_{1}(4)}{\sigma_{2}^{2}}   &     \kappa_{4}-\frac{(\sigma_{1}(4))^{2}}{\sigma_{2}^{2}}
\end{matrix}
\right).
\end{equation*}

\bigskip

\noindent
{\bf Step 2.} For $i=1,2,3$, multiplying the $i$-th row by $\sigma_{1}(4)$, we obtain
\begin{equation*}
\left(
\begin{matrix}
\frac{\sigma_{1}(4)\sigma_{2}\kappa_{4}-\sigma_{1}(1)}{\sigma_{2}}   &     \frac{\sigma_{1}(4)-\sigma_{1}(1)}{\sigma_{2}}                      & \frac{\sigma_{1}(4)-\sigma_{1}(1)}{\sigma_{2}}                       &     \frac{\sigma_{1}(4)-\sigma_{1}(1)\sigma_{2}\kappa_{4}}{\sigma_{2}}  \\[3mm]
\frac{\sigma_{1}(4)-\sigma_{1}(2)}{\sigma_{2}}                       &     \frac{\sigma_{1}(4)\sigma_{2}\kappa_{4}-\sigma_{1}(2)}{\sigma_{2}}  &
\frac{\sigma_{1}(4)-\sigma_{1}(2)}{\sigma_{2}}                       &     \frac{\sigma_{1}(4)-\sigma_{1}(2)\sigma_{2}\kappa_{4}}{\sigma_{2}}  \\[3mm]
\frac{\sigma_{1}(4)-\sigma_{1}(3)}{\sigma_{2}}                       &     \frac{\sigma_{1}(4)-\sigma_{1}(3)}{\sigma_{2}}                      &
\frac{\sigma_{1}(4)\sigma_{2}\kappa_{4}-\sigma_{1}(3)}{\sigma_{2}}   &   \frac{\sigma_{1}(4)-\sigma_{1}(3)\sigma_{2}\kappa_{4}}{\sigma_{2}}    \\[3mm]
\frac{\sigma_{2}-\sigma_{1}(4)\sigma_{1}(1)}{\sigma_{2}^{2}}         &     \frac{\sigma_{2}-\sigma_{1}(4)\sigma_{1}(2)}{\sigma_{2}^{2}}        &
\frac{\sigma_{2}-\sigma_{1}(3)\sigma_{1}(4)}{\sigma_{2}^{2}}         &     \kappa_{4}-\frac{(\sigma_{1}(4))^{2}}{\sigma_{2}^{2}}
\end{matrix}
\right).
\end{equation*}

\bigskip

\noindent
{\bf Step 3.} For $i=2,3$, multiplying the $1$-st row by $-\frac{\sigma_{1}(4)-\sigma_{1}(i)}{\sigma_{1}(4)-\sigma_{1}(1)}$, and adding that to the $i$-th row, we obtain
\begin{equation*}
\left(
\begin{matrix}
\frac{\sigma_{1}(4)\sigma_{2}\kappa_{4}-\sigma_{1}(1)}{\sigma_{2}}   &     \frac{\sigma_{1}(4)-\sigma_{1}(1)}{\sigma_{2}}                      & \frac{\sigma_{1}(4)-\sigma_{1}(1)}{\sigma_{2}}                       &     \frac{\sigma_{1}(4)-\sigma_{1}(1)\sigma_{2}\kappa_{4}}{\sigma_{2}}  \\[3mm]
 a_{21}                                                              &      a_{22}                                                             &
 0                                                                   &      a_{24}                                                             \\[3mm]
 a_{31}                                                              &      0                                                                  &
 a_{33}                                                              &      a_{34}                                                             \\[3mm]
\frac{\sigma_{2}-\sigma_{1}(4)\sigma_{1}(1)}{\sigma_{2}^{2}}         &     \frac{\sigma_{2}-\sigma_{1}(4)\sigma_{1}(2)}{\sigma_{2}^{2}}        &
\frac{\sigma_{2}-\sigma_{1}(3)\sigma_{1}(4)}{\sigma_{2}^{2}}         &     \kappa_{4}-\frac{(\sigma_{1}(4))^{2}}{\sigma_{2}^{2}}
\end{matrix}
\right),
\end{equation*}
where
\begin{equation*}
\begin{split}
a_{i1} & = \frac{\sigma_{1}(4)-\sigma_{1}(i)}{\sigma_{2}}
           -\frac{\sigma_{1}(4)-\sigma_{1}(i)}{\sigma_{1}(4)-\sigma_{1}(1)}\cdot\frac{\sigma_{1}(4)\sigma_{2}\kappa_{4}-\sigma_{1}(1)}{\sigma_{2}},\\
a_{ii} & = \sigma_{1}(4)\kappa_{4}-\frac{\sigma_{1}(4)}{\sigma_{2}},\\
a_{i4} & = \frac{\sigma_{1}(4)}{\sigma_{2}}-\sigma_{1}(i)\kappa_{4}
           -\frac{\sigma_{1}(4)-\sigma_{1}(i)}{\sigma_{1}(4)-\sigma_{1}(1)}\cdot\frac{\sigma_{1}(4)-\sigma_{1}(i)\sigma_{2}\kappa_{4}}{\sigma_{2}},
\end{split}
\end{equation*}
for $i=2,3$.

\bigskip

\noindent
{\bf Step 4.} For $i=2,3$, multiplying the $i$-th row by $-\frac{\sigma_{2}-\sigma_{1}(i)\sigma_{1}(4)}{\sigma_{2}^{2}a_{ii}}$, and adding that to the $4$-th row, we obtain
\begin{equation*}
\left(
\begin{matrix}
\frac{\sigma_{1}(4)\sigma_{2}\kappa_{4}-\sigma_{1}(1)}{\sigma_{2}}   &     \frac{\sigma_{1}(4)-\sigma_{1}(1)}{\sigma_{2}}                      & \frac{\sigma_{1}(4)-\sigma_{1}(1)}{\sigma_{2}}                       &     \frac{\sigma_{1}(4)-\sigma_{1}(1)\sigma_{2}\kappa_{4}}{\sigma_{2}}  \\[3mm]
 a_{21}                                                              &      a_{22}                                                             &
 0                                                                   &      a_{24}                                                             \\[3mm]
 a_{31}                                                              &      0                                                                  &
 a_{33}                                                              &      a_{34}                                                             \\[3mm]
 a_{41}                                                              &      0                                                                  &
 0                                                                   &      a_{44}
\end{matrix}
\right),
\end{equation*}
where
\begin{equation*}
\begin{split}
a_{41} & = \frac{\sigma_{2}-\sigma_{1}(1)\sigma_{1}(4)}{\sigma_{2}^{2}}
           -\sum_{i=2}^{3}\frac{\sigma_{2}-\sigma_{1}(i)\sigma_{1}(4)}{\sigma_{2}^{2}a_{ii}}a_{i1},\\
a_{44} & = \kappa_{4}-\frac{(\sigma_{1}(4))^{2}}{\sigma_{2}^{2}}
           -\sum_{i=2}^{3}\frac{\sigma_{2}-\sigma_{1}(i)\sigma_{1}(4)}{\sigma_{2}^{2}a_{ii}}a_{i4}.
\end{split}
\end{equation*}

\bigskip

By  the part  (1) of Lemma \ref{Lemma 6 and Lemma 7},   a direct calculation shows
\begin{equation}\label{The order of aij}
\begin{split}
|a_{i1}| \leq C_{A}, & ~a_{ii}=O(\lambda_{1}),~a_{i4}=O(\lambda_{1}) \text{~for $i=2,3$,}\\
|a_{41}| & \leq  C\lambda_{1},~a_{44}=O(\lambda_{1}^{2}),
\end{split}
\end{equation}
where $O(\lambda_{1}^{s})$ denotes a term satisfying $C_{A}^{-1}\lambda_{1}^{s}\leq |O(\lambda_{1}^{s})|\leq C_{A}\lambda_{1}^{s}$. Moreover,  we see that
\begin{equation*}
\vec{d}=(d_{1},d_{2},d_{3},d_{4}) = \left(1,\frac{a_{24}a_{41}-a_{21}a_{44}}{a_{22}a_{44}},\frac{a_{34}a_{41}-a_{31}a_{44}}{a_{33}a_{44}},-\frac{a_{41}}{a_{44}}\right)
\end{equation*}
is an eigenvector of matrix $(-G^{i\overline{i},j\overline{j}})$ corresponding to $\kappa_{4}$. By (\ref{The order of aij}), we obtain $|d_{i}|\leq\frac{C_{A}}{\lambda_{1}}$ for $i=2,3,4$. Thus
 $\xi_{4}=\frac{\vec{d}}{\|\vec{d}\|}$ and each $|\xi_{4}^i|^2\le C_A\lambda_1^{-2},~i=2,3,4.$
 The proof of (2) in Lemma \ref{Lemma 6 and Lemma 7} is proved.
\end{proof}

\end{document}